\documentclass[a4paper,11pt]{article}
\usepackage{graphicx}
\usepackage{amsmath}
\usepackage{amsfonts}
\usepackage{tikz-cd}
\usepackage{booktabs}
\usepackage{mathrsfs}
\usepackage{multirow}
\usepackage{multicol}
\usepackage{float}
\usepackage{caption}
\usepackage{amssymb}
\usepackage{xcolor}
\usepackage{hyperref}
\usepackage{faktor}
\usepackage[
    backend=biber,
    style=numeric,
    maxbibnames=999
]{biblatex}

\DeclareFieldFormat[
    article,
    inbook,
    incollection,
    inproceedings,
    patent,
    thesis,
    unpublished
]{title}{#1\isdot}

\newenvironment{proof}[1][Proof]{\noindent\textbf{#1.} }{\
\rule{0.5em}{0.5em}}

\newtheorem{theorem}{Theorem}[section]

\newtheorem{algorithm}[theorem]{Algorithm}

\usepackage{algorithm}
\usepackage{algpseudocode}

\newtheorem{corollary}[theorem]{Corollary}

\newtheorem{definition}[theorem]{Definition}
\newtheorem{example}[theorem]{Example}

\newtheorem{lemma}[theorem]{Lemma}
\newtheorem{notation}[theorem]{Notation}

\newtheorem{proposition}[theorem]{Proposition}

\newtheorem{remark}[theorem]{Remark}

\newenvironment{theoremA}
  {\par\medskip\noindent\textbf{Theorem.}\itshape}
  {\par\medskip}

\newcommand\K{\mathbb K}

\newcommand\N{\mathbb N}

\newcommand\bu{\mathbf u}

\title{New bounds for the support of input-output equations in differential-algebraic systems
\footnote{Partially supported by Universidad de Buenos Aires (UBACYT 20020250100090BA), Argentina.}}
\author{Gabriela Jeronimo{$^{1, 2, 3}$} \and Leonardo Lanciano{$^{1}$}}

\date{}

\begin{document}

\maketitle

\begin{minipage}{14cm}
\noindent {\small $^1$ Universidad de Buenos Aires, Facultad de Ciencias Exactas y Naturales,  Departamento de Matem\'atica. Buenos Aires, Argentina.}

\noindent {\small $^2$  CONICET -- Universidad de Buenos Aires, Instituto de Investigaciones Matem\'aticas ``Luis A. Santal\'o'' (IMAS). Buenos Aires, Argentina.}

\noindent {\small $^3$ Universidad de Buenos Aires, Ciclo B\'asico Com\'un,  Departamento de Ciencias Exactas. Buenos Aires, Argentina.}
\end{minipage}

\bigskip

\begin{abstract}
Given a polynomial dynamical system $\mathbf{x}'=\mathbf{f}(\mathbf{x},\mathbf{u})$ together with an observation function $y=g(\mathbf{x},\mathbf{u})$, where $\mathbf{x}=(x_1,\ldots,x_n)$, $\mathbf{u}=(u_1,\ldots,u_m)$ and $y$ are differential variables, and $\mathbf{f}=(f_1,\ldots,f_n)$, $g$ are polynomials with coefficients in a differential field, we study the problem of determining a minimal polynomial differential equation satisfied by the inputs $\mathbf{u}$ and the output $y$ which follows as a differential consequence of the system. 

We provide a characterization of a finite superset of the set of monomials appearing with non-zero
coefficients in this input-output equation. Specifically, we establish an upper bound for the degree of the
minimal polynomial and a family of inequalities that define a polytope containing its Newton polytope.
These results extend recent work by Mukhina and Pogudin for systems with constant parameters, and enable the use
of evaluation-interpolation techniques for the efficient computation of such eliminant polynomials.
\end{abstract}


\section{Introduction}

The theory of differential algebra initiated by Ritt in the 1930s and followed by Kolchin (see \cite{Ritt1932}, \cite{ritt1950differential}, \cite{kolchin1973differential}) provides a natural framework for studying systems of differential algebraic equations, that is, systems where the equations are given by polynomials in the unknown functions and their derivatives. Since then, and especially in the last decades, substantial progress has been made in understanding structural, quantitative and algorithmic aspects of such systems  from a new algebraic point of view (see, for instance, \cite{Boulier-etal1995}, \cite{Hubert2003}, \cite{DJS2006}, \cite{GKO2016} and the references therein).

A fundamental problem in this area is \emph{differential elimination}: given a system of differential algebraic equations  the goal is to determine differential polynomials that involve only a subset of the variables and are differential-algebraic consequences of the system (more precisely, polynomials lying in a differential ideal generated by the system). Despite important advances following different approaches (see, for instance, \cite{Grigorev1989}, \cite{CarraFerro2007}, \cite{DALFONSO2014588},  \cite{Li-etal2015}, \cite{Rueda2016}, \cite{10.1093/imrn/rnaa302}, \cite{doi:10.1137/22M1469067}), general effective methods for differential elimination remain limited, mainly due to the rapid growth in size of the polynomials involved.

Motivated by applications, including models in systems biology, chemical reaction networks, and engineering systems with inputs and outputs (see, for instance, \cite{Boulier+2007+109+138},  \cite{Fliess1995Implicit}, \cite{Sedoglavic2002}), in this paper we consider dynamical systems of the form:
\begin{equation}\label{ecuacion sistema introduccion}
\Sigma :=
\begin{cases}
\mathbf{x}' = \mathbf{f}(\mathbf{x}, \mathbf{u}), \\[4pt]
y = g(\mathbf{x}, \mathbf{u}),
\end{cases}
\end{equation}
where $\mathbf{x} = (x_1,\ldots,x_n)$, $\mathbf{u} = (u_1,\ldots,u_m)$ are differential indeterminates representing the state and input variables, 
respectively, $y$ is an observable output, and $\mathbf{f} =(f_1,\ldots,f_n), g $ are polynomials with coefficients in a differential field $\mathbb{K}$ of zero characteristic.

We address a specific instance of the effective differential elimination problem: for a system as \eqref{ecuacion sistema introduccion}, 
we aim to obtain a minimal polynomial in the variables $\mathbf{u}, y$, and their successive derivatives that follows as a differential consequence 
of the system, which we denote  by $\operatorname{Elim}(\Sigma,\mathbf{x})$ (see Definition \ref{definición eliminante}).
This problem arises naturally in situations where only partial information about the state variables $\mathbf{x}$ can be measured, while the inputs $\mathbf{u}$ and the output $y$ are accessible through measurements. In these situations, determining relations between $\mathbf{u}$ and $y$ that do not involve the state variables $\mathbf{x}$ and every solution of the system should satisfy enables a better understanding of the system and its solutions.  Such \emph{input-output equations} play a central role in applications, for instance in parameter identifiability (see, for instance, \cite{MRS2018}, \cite{Ovchinnikov2023}, \cite{JS2023} and the references therein).

Recently, in \cite{10.1007/978-3-032-09645-6_15} and \cite{mukhina2025projectingdynamicalsystemssupport}, the problem has been studied in the non-parametric setting and for systems with constant parameters.  The authors prove bounds on the \emph{support} of the eliminant polynomial --linear inequalities that the vectors of exponents of its monomials satisfy-- and design an evaluation-interpolation procedure to compute the eliminant from the candidate support.

The main theoretical contribution of this work is to provide explicit bounds
for the support of the eliminant polynomial  $\operatorname{Elim}(\Sigma,\mathbf{x})$ in the setting where the input
parameters $\mathbf{u}$ are not assumed to be constant, but are allowed to vary freely, reflecting the presence of time-dependent external signals or controls.
Our main result is the following:

\medskip

\begin{theoremA}
Let $g,f_1,\ldots,f_n \in \mathbb{K}[\mathbf{x},\mathbf{u}]$ and  $\Sigma$ be the dynamical system introduced in \eqref{def: difsystem}. Let $\nu$ be the order of $\operatorname{Elim}(\Sigma,\mathbf{x})$. 
Set $d:= \deg(g)$, $d_\mathbf{x}:= \deg_\mathbf{x}(g)$, $d_\mathbf{u} := \deg_\mathbf{u}(g)$, 
$D := \max_{1 \leq i \leq n} \deg(f_i)$, $D_{\mathbf{x}} := \max_{1 \leq i \leq n} \deg_{\mathbf{x}}(f_i)$, $ D_{\mathbf{u}} := \max_{1 \leq i \leq n} \deg_{\mathbf{u}}(f_i).$

Assume $D_\mathbf{x}, d_\mathbf{x} > 0$ and let $\Delta := \displaystyle \prod_{i=0}^\nu \bigl(d + i(D-1)\bigr)$ and $\Delta_\mathbf{x}:= \displaystyle \prod_{i=0}^\nu \bigl(d_\mathbf{x} + i(D_{\mathbf{x}}-1)\bigr)$.
Then, for every monomial $\displaystyle\Big( \prod_{k=0}^\nu \prod_{j=1}^m (u_j^{(k)})^{\ell_{j,k}} \Big)\prod_{i=0}^\nu (y^{(i)})^{e_i}$
appearing in $\operatorname{Elim}(\Sigma,\mathbf{x})$, we have:
\begin{enumerate}
\item $\displaystyle \sum_{k=0}^\nu \sum_{j=1}^m \ell_{j,k} + \sum_{i=0}^\nu \bigl(d + i(D-1)\bigr) e_i \;\le\; \Delta,$

\item $\displaystyle  \sum_{i=0}^\nu \bigl(d_\mathbf{x} + i(D_{\mathbf{x}}-1)\bigr) e_i \;\le\; \Delta_\mathbf{x},$

\item $\displaystyle \sum\limits_{k=0}^\nu \sum\limits_{j=1}^m \ell_{j,k} + \sum\limits_{i=0}^\nu \bigl(d_\mathbf{u} + iD_\mathbf{u}\bigr) e_i 
\;\le\;
\sum\limits_{l=0}^\nu \frac{(d_\mathbf{u}+l D_\mathbf{u})\Delta_\mathbf{x}}{d_\mathbf{x} + l(D_\mathbf{x}-1)},$

\item $\displaystyle \sum\limits_{j=1}^m \ell_{j,r}+ \sum\limits_{i=r}^\nu \left \lfloor\frac{i}{r} \right \rfloor e_i
\;\le\; 
\sum\limits_{l=r}^\nu \frac{\left \lfloor{l}/{r} \right \rfloor \Delta_\mathbf{x}}{d_\mathbf{x}+l(D_\mathbf{x}-1)}, $ for $1\le r \le \nu.$ \label{eq:newbounds}
\end{enumerate}
\end{theoremA}

This result can be regarded as a generalization of \cite[Theorem~1]{10.1007/978-3-032-09645-6_15}.  The first inequalities are  natural extensions of that result to the setting of differential parameters. In addition, the bounds corresponding to inequalities of type \ref{eq:newbounds} are new and do not arise as generalizations of previously known results.
When restricted to the constant parameter case $(\mathbf{u}'=0)$, our result yields slightly sharper bounds than those obtained in the existing literature and provides a natural candidate for the Newton polytope of the eliminant in sufficiently generic situations.

Our approach to proving the  new bounds for the support of $\operatorname{Elim}(\Sigma,\mathbf{x})$ is based on  an analogue of the classical Perron's elimination theorem (see Proposition \ref{Coro: Dandrea} below), which is in turn a generalization of \cite[Theorem 3.1]{D2013} (see also \cite[Theorem 3.3]{Jelonek2005}). This approach offers a unified framework that enables us to recover most of the bounds available in the literature
and yields new ones that have no counterpart in the constant-parameter case. 

From our main Theorem, we obtain a candidate support for  $\operatorname{Elim}(\Sigma,\mathbf{x})$ in the differential parameter framework. This enables us to extend the evaluation-interpolation procedure proposed by Mukhina and Pogudin for computing eliminant polynomials to this setting, and perform numerical experiments to assess the quality of the bounds obtained. 

The paper is organized as follows. Section~\ref{sec} introduces the basic notation, definitions, and tools used throughout the paper. In Section~\ref{sec 2}, we present our object of interest, $\operatorname{Elim}(\Sigma,\mathbf{x})$, we then prove the main result of the paper,  compare it with previous work (Subsect.~\ref{sec:main bounds}), and further analyze the case where the output depends on a single state variable,  obtaining refined bounds  (Subsect.~\ref{subsec: output}).  In Section~\ref{sec:experiments}, we present experimental results illustrating the sharpness of our bounds. Finally, Section \ref{sec:conclusions} summarizes the main outcomes and outlines open questions for future research.

\paragraph{Acknowledgements.}
The authors are grateful to Pablo Solernó for stimulating discussions and valuable insights, and to Yulia Mukhina for her useful suggestions regarding the computational experiments.

\section{Preliminaries}\label{sec}
\subsection{Basic definitions and notation}\label{subsect: Basic Def}

Let $(\mathbb{K}, \partial)$ be a differential field of characteristic zero (for instance, $\mathbb{K}=\mathbb{C}$ with $\partial \equiv 0$, or $\mathbb{K} = \mathbb{C}(t)$ with the standard derivation induced by $\partial(t) = 1$). For $n, m \in \mathbb{N}$ we fix differential indeterminates
$
\mathbf{x} = (x_1,\ldots,x_n)$,
$\mathbf{u} = (u_1,\ldots,u_m)$.
We denote the $k$-th successive derivatives of the variables $\mathbf{x}$ by $\mathbf{x}^{(k)} := (x_1^{(k)},\ldots,x_n^{(k)}),$ and similarly for $\mathbf{u}$ (the first derivatives may also be written as $'$). 
The \emph{ring of differential polynomials} in $\mathbf{x}$ and $\mathbf{u}$ with coefficients in $\mathbb{K}$ is defined as the polynomial ring generated by all derivatives of the variables
$\mathbf{x}$ and $\mathbf{u}$,
\[
\mathbb{K}\{\mathbf{x},\mathbf{u}\}:=\mathbb{K}\!\big[\mathbf{x}^{(k)},\, \mathbf{u}^{(k)} : k \in \mathbb{N}_0 \big],
\]
where we extend the derivation of $\mathbb{K}$ by setting $\partial (x_i^{(k)}) = x_i^{(k+1)}$ and $\partial( u_j^{(k)}) = u_j^{(k+1)}$ for every $1\le i \le n$, $1\le j \le m$ and $k\in \mathbb{N}_0$.
For $r,s \in \mathbb{N}$, we write
\[
\mathbf{x}^{(\leq r)} := \big(x_i^{(k)} : 0 \leq k \leq r,\ 1 \le i \le n \big),
\qquad
\mathbf{u}^{(\leq s)} := \big(u_j^{(k)} : 0 \leq k \leq s,\ 1 \le j \le m \big).
\]

The \emph{order  in the variables $\mathbf{x}$} of a differential polynomial $F\in \mathbb{K}\{\mathbf{x}, \mathbf{u}\}$, denoted by $\operatorname{ord}_{\mathbf{x}}(F)$, is
the largest $r\in  \N_0$ such that a variable $x_i^{(r)}$ appears in $F$ (or $-\infty$ if it does not depend on these variables); similarly, we define the order of $F$ in $\mathbf{u}$. Let
$$F=\sum_{I,J}a_{I,J}  (\mathbf{x}^{(\leq r)})^I  (\mathbf{u}^{(\leq s)})^J
\in \mathbb{K}[\mathbf{x}^{(\leq r)},\mathbf{u}^{(\leq s)}],$$
where $I\in\mathbb{N}_0^{n(r+1)}$ and
$J\in\mathbb{N}_0^{m(s+1)}$. 
The total derivative of $F$ in $\mathbb{K}\{\mathbf{x}, \mathbf{u}\}$, denoted by $F'$, is 
    \begin{align*}
    F'
    =
    \sum_{I,J}\partial(a_{I,J})
    (\mathbf{x}^{(\leq r)})^I
    (\mathbf{u}^{(\leq s)})^J
    +\sum_{k=0}^{r}
    \nabla_{\mathbf{x}^{(k)}}(F)\cdot\mathbf{x}^{(k+1)}
    +\sum_{l=0}^{s}
    \nabla_{\mathbf{u}^{(k)}}(F)\cdot\mathbf{u}^{(k+1)}.
    \end{align*}
    Here, the first sum accounts for the action of $\partial$ on the
    coefficients of $F$, which we will denote $\partial_{\text{coeff}}(F)$ in the sequel, and $\nabla_{\mathbf{x}^{(k)}}(F)$ and $\nabla_{\mathbf{u}^{(k)}}(F)$
stand for the vectors of partial derivatives of $F$ with respect to the variables $\mathbf{x}^{(k)}$ and $\mathbf{u}^{(k)}$ respectively.
For $	q\ge 2$,  we write $F^{(q)}$ for the $q$-th total derivative of $F$.

We denote by
    $\deg_{\mathbf{x}^{(\leq r)},\mathbf{u}^{(\leq s)}}(F)$
    the total degree of $F$ in the variables
    $\mathbf{x}^{(\leq r)}$ and $\mathbf{u}^{(\leq s)}$.

The support of $F$ is
    $$\operatorname{supp}(F):=
   \left\{(I,J)\in
    \mathbb{N}_0^{n(r+1)}\times\mathbb{N}_0^{m(s+1)}
    \mid a_{I,J}\neq 0\right\},$$
and its \emph{Newton polytope}, denoted by $\operatorname{Newt}(F)$,
    is the convex hull of $\operatorname{supp}(F)$ in
    $\mathbb{R}^{n(r+1)+m(s+1)}$.

\bigskip

Given a finite family of differential polynomials $
H_1,\ldots,H_\beta \in \mathbb{K}\{\mathbf{x},\mathbf{u}\},
$
we write $[H_1,\ldots,H_\beta]$ to denote the \emph{differential ideal} they generate, that is,
the smallest ideal of $\mathbb{K}\{\mathbf{x},\mathbf{u}\}$ containing $H_1,\ldots,H_\beta$ and
closed under derivation.

\subsection{Some auxiliary facts from algebraic geometry}

Fix \(N \in \mathbb{N}\), and let \(K\) be an algebraic closure of \(\mathbb{K}\). We denote by \(\mathbb{A}^N_K := K^N\) the affine \(N\)-space over \(K\), endowed with the Zariski topology. Unless explicitly stated otherwise, all affine spaces \(\mathbb{A}^N\) are understood as \(\mathbb{A}^N_K\). Given an algebraic variety \(V \subseteq \mathbb{A}^N\), we denote by \(I(V)\) the ideal of all polynomials in $N$ variables with coefficients in $K$ that vanish on \(V\).

\medskip

Now we briefly recall the definition of the \emph{geometric degree} of an algebraic variety and its basic properties. For further results, see, for instance, \cite{mumford1976algebraic,harris1992algebraic} in the case of projective varieties or \cite{Heintz1983} for affine (and constructible) sets.
We also recall the notion of \emph{mixed heights} introduced in \cite{D2013}. These ingredients allow us to prove a modified version of the analogue of Perron's theorem given in \cite[Theorem~3.1]{D2013}, which will be a key ingredient in our approach (see Proposition \ref{Coro: Dandrea} below).

\begin{definition}
Let $V \subseteq \mathbb{A}^N_K$ be an irreducible algebraic variety with $\dim(V) = r$. The \emph{geometric degree} of $V$, denoted by $\deg(V)$, is defined as
\[
\deg(V) := \#(V \cap L),
\]
where $L$ is a generic linear algebraic variety of dimension $N-r$. For an arbitrary algebraic variety $V$, its geometric degree is defined as the sum of the degrees of its irreducible components.
\end{definition}
For a proof that this invariant is well defined, see \cite[Lemma~1, Proposition~1]{Heintz1983}. It is easy to see that the degree of any linear variety is $1$ and the degree of a hypersurface agrees with the degree of the square-free equation defining it.

\bigskip

Let $\mathbf{t}=(t_1,\ldots,t_s)$ be a set of variables and let
$V \subseteq \mathbb{A}^N_{\overline{K(\mathbf{t})}}$ be an irreducible $K(\mathbf{t})$-variety.
The geometric degree of $V$ measures its complexity with respect to the variables $\mathbf{x}$.
In effective algebraic geometry, one is often also interested in measuring the complexity of $V$ with respect to the parameter variables $\mathbf{t}$, which leads to the notion of \emph{height}.  In our setting, the height of $V$, denoted $h(V)$, can be defined as the $\mathbf{t}$-degree of the Chow form of its closure in $\mathbb{P}^N(\overline{K(\mathbf{t})})$.

It is immediate that the height of the affine space \(\mathbb{A}^N\) is zero. Moreover, for a $K(\mathbf{t})$-hypersurface, its height coincides with the total degree in the parameter variables \(\mathbf{t}\) of its primitive square-free defining polynomial with coefficients in $K[\mathbf{t}]$.  We refer the reader to \cite[Sect. 2.1]{D2013} for alternative characterizations and basic properties of the height of varieties over a function field.

If the parameters  are partitioned into blocks
$\mathbf{t}=(\mathbf{t}_1,\ldots,\mathbf{t}_k)$, where $\mathbf{t}_i=(t_{i,1},\dots, t_{i,s_i})$ for $1\le i \le k$, the height of $V \subseteq \mathbb{A}^N_{\overline{K(\mathbf{t})}}$  with respect to the block $\mathbf{t}_i$ is  defined by considering the field ${K}_{\mathbf{t}_i}:=K(\mathbf{t}_j \mid j \neq i)$ and $V$ as a ${K}_{\mathbf{t}_i}(\mathbf{t}_i)$-variety.

\bigskip

We now present a generalization of Perron's elimination theorem (see \cite[Theorem 3.3]{Jelonek2005}), which follows essentially from 
\cite[Theorem 3.1 and Corollary 3.14]{D2013} and yields a similar conclusion under hypotheses better adapted to our setting.
\begin{proposition}\label{Coro: Dandrea}
 Let $\mathbf{t}_i=(t_{i,1}, \dots, t_{i,s_i})$, $1\le i \le k$, be blocks of variables and $\mathbf{t}=(\mathbf{t}_1,\ldots,\mathbf{t}_k)$, and let
$V \subseteq \mathbb{A}^N$ be an irreducible $K(\mathbf{t})$-variety. Let $q_0,\ldots,q_{\nu} \in K[\mathbf{t},\mathbf{x}] \setminus K[\mathbf{t}]$ and assume that the Zariski closure of the image of the map
\[
\varphi : V \rightarrow \mathbb{A}^{\nu+1}, \qquad
\mathbf{x} \mapsto (q_0(\mathbf{x}),\ldots,q_{\nu}(\mathbf{x})),
\]
is a hypersurface. Set $d_j \geq \deg_{\mathbf{x}}(q_j)$ and $b_{i,j} \geq \deg_{\mathbf{t}_i}(q_j)$ for $0\le j\le \nu$, $1\le i\le k$, and write
$\mathbf{d}=(d_0,\ldots,d_{\nu})$ and $\mathbf{b}_i=(b_{i,0},\ldots,b_{i,\nu})$ for $1\le i \le k$.  Then, there exists a primitive irreducible polynomial
\[
E = \sum_{I\in\mathbb{N}_0^{\nu+1}} \alpha_I\, \mathbf{Y}^I \in K[\mathbf{t}][Y_0,\ldots,Y_\nu]
\]
defining $\overline{\operatorname{Im}(\varphi)}$ such that, for all $I\in\operatorname{supp}(E)$ the following inequalities hold:
\begin{itemize}
\item $\displaystyle \langle \mathbf{d}, I \rangle \;\le\; \deg(V)  \Big ( \prod_{j=0}^{\nu} d_j  \Big),$
\item $ \displaystyle \deg_{\mathbf{t}_i}(\alpha_I)+\langle \mathbf{b}_i, I\rangle
\;\le\; \Big ( \prod_{j=0}^{\nu} d_j \Big) \left ( h_{\mathbf{t}_i}(V) + \deg(V) \Big ( \sum_{\ell=0}^{\nu} \frac{b_{i,\ell}}{d_\ell} \Big) \right )
\textit{ for } 1\le i \le k.$
\end{itemize}
\end{proposition}

\begin{proof}
Under our assumptions, we set
$W =  \overline{\operatorname{Im}(\varphi)},$ the hypersurface defined in $\mathbb{A}^{\nu +1}$ as the zero set of the polynomial $E$.
Since the map $\varphi : V \to W$ is dominant, the induced map of function fields
$K(\mathbf{t})(W) \hookrightarrow K(\mathbf{t})(V)$ is an inclusion of fields.
Moreover, $K(\mathbf{t})(W)$ is generated over $K(\mathbf{t})$ by the classes of
$q_0,\ldots,q_{\nu}$. As $
\operatorname{trdeg}_{K(\mathbf{t})} K(\mathbf{t})(W) = \nu$, without loss of generality, we may assume that
$q_1,\ldots,q_\nu$ form a transcendence basis of $K(\mathbf{t})(W)$ over $K(\mathbf{t})$.

On the other hand, $K(\mathbf{t})(V)$ is generated over $K(\mathbf{t})$ by the classes of
$x_1,\ldots,x_N$, and if $r=\dim(V)$, we have
$\operatorname{trdeg}_{K(\mathbf{t})} K(\mathbf{t})(V) = r.$
It follows that $K(\mathbf{t})(V)$ has transcendence degree $r-\nu$ over
$K(\mathbf{t})(W)$, and we may assume that the classes of $x_1,\ldots,x_{r-\nu}$ form a transcendence basis over $K(\mathbf{t})(W)$.
Therefore, $q_1,\ldots, q_\nu, x_1, \dots, x_{r-\nu}$ is a transcendence basis of $K(\mathbf{t})(V)$ over $K(\mathbf{t})$.

We now consider the polynomials $q_0,q_1\ldots,q_{\nu},x_1,\ldots,x_{r-\nu}$, and the map
\[\psi: V \to \mathbb{A}^{r+1}, \qquad \mathbf{x} \mapsto (q_0(\mathbf{x}),\ldots,q_{\nu}(\mathbf{x}), x_1,\dots, x_{r-\nu}).\]
Since $q_1,\ldots,q_\nu,x_1,\ldots,x_{r-\nu}$ are algebraically independent over $K(\mathbf{t})$ and $V$ is irreducible, the Zariski closure of the image $\psi$ is a hypersurface and $\psi$ is generically finite onto its image. 

A careful analysis of the proof of \cite[Theorem 3.1]{D2013} implies that the result is also valid  when the \(d_j\) and \(b_{i,j}\) are taken as upper bounds for the corresponding degrees instead of the exact degrees. Then, applying  \cite[Corollary 3.14]{D2013}, we deduce that a primitive irreducible polynomial
\[
\hat{E} =
 \sum_{(I, J )\in\mathbb{N}_0^{\nu+1}\times \mathbb{N}_0^{r-\nu}} \alpha_{I, J}\, \mathbf{Y}^I  \mathbf{X}^J 
\in K[\mathbf{t}][Y_0,\ldots,Y_{\nu},X_1,\ldots,X_{r-\nu}]
\]
such that $\hat{E}(q_0,\ldots,q_{\nu},x_1,\ldots,x_{r-\nu})=0$ on $V$ satisfies the following support bounds: for every $(I,J)\in \operatorname{supp}(\hat{E}) \subseteq \mathbb{N}_0^{\nu+1} \times \N_0^{r-\nu}$,
\begin{itemize}
\item $\langle \mathbf{d},I \rangle + \langle \mathbf{1},J \rangle \le \displaystyle \deg(V) \Big(\prod_{j=0}^{\nu} d_j \Big) \Big( \prod_{h=1}^{r-\nu} 1 \Big) = \displaystyle \deg(V) \Big ( \prod_{j=0}^{\nu} d_j \Big),$
\item for $1\le i \le k$, 
\begin{align*}
\deg_{\mathbf{t}_i}(\alpha_{I,J}) + \langle \mathbf{b}_i,I\rangle + \langle \mathbf{0},J \rangle 
& \le \displaystyle \Big( \prod_{j=0}^{\nu} d_j \Big) \Big ( \prod_{h=1}^{r-\nu} 1 \Big)  \Big(h_{\mathbf{t}_i}(V) + \deg(V) \Big( \sum_{\ell=0}^{\nu} \frac{b_{i,\ell}}{d_\ell} + \sum_{h=1}^{r-\nu} \frac{0}{1} \Big) \Big) \\ 
& = \displaystyle \Big( \prod_{j=0}^{\nu} d_j \Big)  \Big( h_{\mathbf{t}_i}(V) + \deg(V) \Big( \sum_{\ell=0}^{\nu} \frac{b_{i,\ell}}{d_\ell}  \Big) \Big).
    \end{align*}
\end{itemize}
As the defining equation of a hypersurface is unique -up to multiplication by a nonzero scalar- and the irreducible polynomial $E$ vanishes on the image of $\psi$ (since $E(q_0,\dots, q_\nu)=0$ on $V$), it follows that $\hat{E}$ coincides with $E$. In particular, this implies that $J=0$ for every $(I,J)\in \operatorname{supp}(\hat{E})$, and the bounds in the statement  hold. \end{proof}

\section{Bounds for the Support of the Eliminant Polynomial} \label{sec 2}

Let $\mathbf{x}=(x_1,\ldots,x_n)$, $\mathbf{u}=(u_1,\ldots,u_m)$ and $y$ be differential indeterminates over $\mathbb{K}$, as introduced in Subsection \ref{subsect: Basic Def}. Given polynomials $ f_1,\ldots,f_n, g \in \mathbb{K}[\mathbf{x},\mathbf{u}]$, we consider the dynamical system with one observation function $y$ given by:
\begin{equation}\label{def: difsystem}
\Sigma :=
\begin{cases}
\mathbf{x}' = \mathbf{f}(\mathbf{x},\mathbf{u}), \\[4pt]
y = g(\mathbf{x},\mathbf{u}),
\end{cases}
\end{equation}
where $\mathbf{f} = (f_{1},\ldots,f_{n})$.
Associated to this dynamical system we have the differential ideal of $\mathbb{K}\{\mathbf{x},\mathbf{u},y\}$ generated by the differential polynomials $\mathbf{x}_i'-f_i(\mathbf{x},\mathbf{u})$, for $1\le i \le n$, and $y- g(\mathbf{x},\mathbf{u})$. We will write $[\Sigma] \subseteq \mathbb{K}\{\mathbf{x},\mathbf{u},y\}$ to denote this associated differential ideal. 

Our aim is to characterize a minimal differential polynomial in the ideal $[\Sigma]$, in the sense that we will precise later, depending only on the differential variables $\mathbf{u}$ and $y$ (and its successive derivatives), that is, in $[\Sigma]\cap \mathbb{K}\{ \mathbf{u},y\} $.
This is an \emph{eliminant} polynomial, since we eliminate the state variables $\mathbf{x}$ to get a differential equation involving only the input and output variables ($\mathbf{u}$ and $y$).

\subsection{Eliminant polynomials}

The goal of this subsection is to define an eliminant polynomial associated with the system $\Sigma$.
This eliminant will be obtained from the algebraic relations among
$y, y', \ldots, y^{(k)},\dots $ arising from the identity $y= g(\mathbf{x}, \mathbf{u})$.

By formally differentiating the observation function in $\mathbb{K}[\mathbf{x},\mathbf{u}]$,
we obtain a sequence of polynomials
$(g^{(k)})_{k\ge 0} \subseteq
\mathbb{K}\{\mathbf{x},\mathbf{u}\}$
defined recursively from $g(\mathbf{x},\mathbf{u})$ by:
\begin{align*}
 g'(\mathbf{x},\mathbf{u},\mathbf{x}',\mathbf{u}') & =  \partial_{\operatorname{coef}}(g)+\nabla_{\mathbf{x}}(g)\cdot \mathbf{x}'
       +  \nabla_{\mathbf{u}}(g)\cdot \mathbf{u}', \\[6pt]
g^{(2)}\big(\mathbf{x}^{(\leq 2)},\mathbf{u}^{(\leq 2)}\big)  & = \partial_{\operatorname{coef}}(g') +  \nabla_{\mathbf{x}}(g')\cdot \mathbf{x}'
       + \nabla_{\mathbf{x}'}(g')\cdot \mathbf{x}^{(2)}
       + \nabla_{\mathbf{u}}(g')\cdot \mathbf{u}'
       + \nabla_{\mathbf{u}'}(g')\cdot \mathbf{u}^{(2)}, \\[6pt]
& \vdots \\
g^{(k)}\big(\mathbf{x}^{(\leq k)},\mathbf{u}^{(\leq k)}\big)   & = \; \partial_{\operatorname{coef}}(g^{(k-1)})+
\sum_{i=0}^{k-1} \nabla_{\mathbf{x}^{(i)}}(g^{(k-1)}) \cdot \mathbf{x}^{(i+1)}
+
\sum_{i=0}^{k-1} \nabla_{\mathbf{u}^{(i)}}(g^{(k-1)}) \cdot \mathbf{u}^{(i+1)}\\
& \vdots
\end{align*}
such that, according to $\Sigma$,  for every $k\ge 0$, we have
$$y^{(k)} =g^{(k)}\big(\mathbf{x}^{(\leq k)},\mathbf{u}^{(\leq k)}\big).$$

Consider the $\K\{\bu\}$-algebra (evaluation) morphism
\begin{center}
$\mathcal{E}: \K\{\mathbf{x},\bu,y\} \to \K\{\bu\}[\mathbf{x}]$,
\end{center}
defined by $\mathcal{E}(u_j^{(k)}) =u_j^{(k)}$, $\mathcal{E}(x_i) =x_i$, $\mathcal{E}(y) = g(\mathbf{x}, \bu)$, $\mathcal{E}(x_i')= f_i(\mathbf{x}, \bu)$
and, recursively, for $k>0$,
$\mathcal{E}(x_i^{(k)})= \mathcal{E}(f_i^{(k-1)})$ for  $i=1,\dots, n$,  and $\mathcal{E}(y^{(k)}) = \mathcal{E}(g^{(k)}).$

The morphism $\mathcal{E}$ allows us to encode the differential system algebraically by substituting the derivatives of the state variables according to the equations of $\Sigma$. 
It follows straightforwardly that
$$\K\{\mathbf{x}, \bu,y\}/[\Sigma]\simeq \K\{\bu\}[\mathbf{x}],$$
and then, $\{y,y',\dots, y^{(n)}\}\subset \K\{\mathbf{x}, \bu,y\}/[\Sigma]$ is algebraically dependent over $\text{Frac}(\K\{\bu\})$. Let  $\nu\le n$ be the minimum such that $\{y,y',\dots, y^{(\nu)}\}$ is algebraically dependent.

By restricting $\mathcal{E}$ to $\mathbb{K}[\mathbf{x}^{(\leq \nu)},\mathbf{u}^{(\leq \nu)}]$, we 
obtain the $\mathbb{K}[\mathbf{u}^{(\leq \nu)}]$-algebra morphism
\[
\mathcal{E}:
\mathbb{K}[\mathbf{x}^{(\leq \nu)},\mathbf{u}^{(\leq \nu)}]
\rightarrow
\mathbb{K}[\mathbf{x},\mathbf{u}^{(\leq \nu)}]
\]
and, extending scalars to $\mathbb{K}(\mathbf{u}^{(\leq \nu)})$,  we have a $\mathbb{K}(\mathbf{u}^{(\leq \nu)})$-algebra morphism
\[
\mathcal{E}:
\mathbb{K}(\mathbf{u}^{(\leq \nu)})[\mathbf{x}^{(\leq \nu)}]
\rightarrow
\mathbb{K}(\mathbf{u}^{(\leq \nu)})[\mathbf{x}],
\]
where the variables $\mathbf{u}$ and their derivatives are viewed as elements in the coefficient base field.

This leads naturally to the following polynomial maps, whose images will control the algebraic relations among the derivatives of the output $y$ of the system $\Sigma$.

\begin{notation}
Let $\widehat{K}$ be an algebraic closure of $\mathbb{K}(\mathbf{u}^{(\leq \nu)})$. We define morphisms
\[\begin{array}{rcl}
     \varphi : \mathbb{A}^n_{\widehat{K}} &\to& \mathbb{A}^{\nu+1}_{\widehat{K}} \\
     \mathbf{x} &\mapsto &
 \big(\mathcal{E}(g), \mathcal{E}(g'),\ldots, \mathcal{E}(g^{(\nu)})\big)
     \end{array}\qquad
 \begin{array}{rcl}
 \varphi^{\mathbf{u}} : \mathbb{A}^{n + m (\nu+1)}_K
 &\to& \mathbb{A}^{(\nu+1)+m(\nu+1)}_K \\
 (\mathbf{u}^{(\leq \nu)},\mathbf{x}) &\mapsto &
 \big(\mathbf{u}^{(\leq \nu)}, \mathcal{E}(g), \ldots, \mathcal{E}(g^{(\nu)})\big)
 \end{array}\]
\end{notation}

The next proposition shows that both constructions lead to hypersurfaces and that, after a suitable normalization, their defining equations coincide. 

\begin{proposition}\label{prop: dim e irred}
In the previous setting, the Zariski closure of the images of the morphisms $\varphi^{\mathbf{u}}$ and $\varphi$ are hypersurfaces. Moreover, the primitive irreducible polynomial in $K[\mathbf{u}^{(\leq \nu)}][Y_0,\ldots,Y_{\nu}]$ defining $\overline{\operatorname{Im}(\varphi)}$ coincides, up to a scalar factor, with the defining equation of $\overline{\operatorname{Im}(\varphi^{\mathbf{u}})}$.
\end{proposition}

\begin{proof}
Since the elements $y,\ldots,y^{(\nu)}$ are algebraically dependent over $\mathbb{K}(\mathbf{u}^{(\leq \nu)})$ modulo $[\Sigma]$,  there exists a nonzero polynomial  $E \in K(\mathbf{u}^{(\leq \nu)})[Y_0,\ldots,Y_\nu]$ such that:
\[
E(y,\ldots,y^{(\nu)}) = 0 \mod [\Sigma].
\]
Multiplying by a suitable element of $K(\mathbf{u}^{(\leq \nu)})$, we may assume that $E$ has polynomial coefficients in $\mathbf{u}^{(\leq \nu)}$. It follows that $E$ is a nonzero polynomial vanishing on $\operatorname{Im}(\varphi^{\mathbf{u}})$ and, therefore,
\[
\dim\big(\overline{\operatorname{Im}(\varphi^{\mathbf{u}})}\big)
\le m(\nu+1)+\nu.
\]
To prove the opposite inequality, let $\pi: \mathbb{A}^{m(\nu+1)} \times \mathbb{A}^{\nu+1}
\rightarrow
\mathbb{A}^{m(\nu+1)} \times \mathbb{A}^{\nu}$
be the projection that forgets the last coordinate. We claim that: $\overline{\pi \circ \varphi^{\mathbf{u}}(\mathbb{A}^{n+m(\nu+1)})}
=
\mathbb{A}^{m(\nu+1)} \times \mathbb{A}^{\nu}.$
Otherwise, there would exist a nonzero polynomial
$
P \in K[\mathbf{u}^{(\leq \nu)},Y_0,\ldots,Y_{\nu-1}]
$
such that $P \circ (\pi \circ \varphi^{\mathbf{u}})$ vanishes identically on $\mathbb{A}^{n + m(\nu+1)}$. Since the first $m(\nu+1)$ coordinates of $\pi \circ \varphi^{\mathbf{u}}$ are precisely $\mathbf{u}^{(\leq \nu)}$, this implies that
$
P\big(\mathbf{u}^{(\leq \nu)}, \mathcal{E}(g),\ldots,\mathcal{E}(g^{(\nu-1)})\big)=0.
$
This yields a polynomial relation
\[
P\big(\mathbf{u}^{(\leq \nu)}, y, y', \ldots, y^{(\nu-1)}\big)=0
\mod [\Sigma],
\]
which contradicts the fact that $\{y,\ldots,y^{(\nu-1)}\}$ are algebraically independent over $\mathbb{K}(\mathbf{u}^{(\leq \nu)})$. Therefore, $\dim\big(\overline{\operatorname{Im}(\varphi^{\mathbf{u}})}\big)
\ge m(\nu+1)+\nu$ and $\overline{\operatorname{Im}(\varphi^\mathbf{u})}$ is a hypersurface.

The statement for $\varphi$ follows analogously.

Now, let $E^{\mathbf{u}} \in K[\mathbf{u}^{(\leq \nu)},Y_0,\ldots,Y_\nu]$ and $E \in K(\mathbf{u}^{(\leq \nu)})[Y_0,\ldots,Y_\nu]$
be the irreducible defining equations of
$\overline{\operatorname{Im}(\varphi^{\mathbf{u}})}$ and $\overline{\operatorname{Im}(\varphi)}$, respectively.
Multiplying $E$ by a suitable element in $K(\mathbf{u}^{(\leq \nu)})^\times$, we may assume that it has coefficients in
$K[\mathbf{u}^{(\leq \nu)}]$ and is primitive.
Then $E$ vanishes on $\overline{\operatorname{Im}(\varphi^{\mathbf{u}})}$, so it lies in the ideal generated by $E^{\mathbf{u}}$. Conversely, viewing $\mathbf{u}^{(\leq \nu)}$ as parameters, $E^{\mathbf{u}}$ vanishes on $\overline{\operatorname{Im}(\varphi)}$, hence it lies in the ideal generated by $E$.

Thus, $E$ and $E^{\mathbf{u}}$ divide each other in $K(\mathbf{u}^{(\leq \nu)})[Y_0,\ldots,Y_\nu]$, and so, they differ by a multiplicative factor in $K(\mathbf{u}^{(\leq \nu)})^\times$. Since both are primitive in $K[\mathbf{u}^{(\leq \nu)}]$, this factor lies in $K^\times$, and the claim follows.
\end{proof}

\bigskip

The previous results allow us to introduce the eliminant polynomial associated with the system $\Sigma$.
  \begin{definition}[Eliminant polynomial] \label{definición eliminante}
Let $\Sigma$ be a polynomial dynamical system as in \eqref{def: difsystem}.
The \emph{eliminant polynomial} of $\Sigma$, obtained by eliminating the state variables $\mathbf{x}$, is the unique (up to a nonzero scalar)
irreducible polynomial
\[
\operatorname{Elim}(\Sigma,\mathbf{x})
\in \K[\mathbf{u}^{(\leq \nu)},\,y,\,y',\ldots,y^{(\nu)}]
\]
introduced in Proposition~\ref{prop: dim e irred} as the defining equation of $\overline{\operatorname{Im}(\varphi^\mathbf{u})}$.
\end{definition}
It follows directly from the definition of \(\nu\) and from
Proposition~\ref{prop: dim e irred} that the polynomial
\(\operatorname{Elim}(\Sigma,\mathbf{x})\) has order exactly \(\nu\).

\bigskip
The next examples illustrate the eliminant polynomial in the nonparametric case and in the presence of a single differential parameter, respectively.

\begin{example}\label{Ejemplo}
Consider
\[
\Sigma:\quad
\begin{cases}
x_1' = x_1^d + x_2^d,\\
x_2' = x_1,\\
y = x_1.
\end{cases}
\qquad d \ge 2.
\]
The corresponding eliminant polynomial is
\[
\operatorname{Elim}(\Sigma,\mathbf{x})
=
\bigl(y'' - d y^{d-1} y'\bigr)^d
- d^d y^d \bigl(y' - y^d\bigr)^{d-1}
\in
\mathbb{K}[y,y',y''].
\]
\end{example}

\begin{example}\label{segundo ejemplo}
Consider a system with a single differential parameter $u$:
\[
\Sigma:\quad
\begin{cases}
x_1' = u + 3x_1 - 5x_2,\\
x_2' = -4 + 2x_1 - 2u,\\
y = 5 + 11x_1 + 3u^2.
\end{cases}
\]
The corresponding eliminant polynomial is
\[
\operatorname{Elim}(\Sigma,\mathbf{x}) =
30u^2 - 18u u' + 6(u')^2 + 6u\,u^{(2)} + 110u + 11u' - 10y + 3y' - y'' + 270
\in \mathbb{K}[u^{(\leq2)},y,y',y''].
\]
\end{example}

\subsection{Degree bounds and proof of the main result}\label{sec:main bounds}

The proof of our main theorem (see Theorem~\ref{mainthm}) is rather short once the appropriate tools are in place. 
The main idea will be to apply Proposition~\ref{Coro: Dandrea} to
$(\mathbb{A}^n,\varphi)$ and $(\mathbb{A}^{n+m(\nu+1)},\varphi^{\mathbf{u}})$.
We therefore begin by proving the upper bounds for the degrees of the polynomials involved in the definition of these maps that we will use.

\begin{lemma}\label{lema tecnico}
Let $f_1,\ldots,f_n, g \in \mathbb{K}[\mathbf{x},\mathbf{u}]$ and  $\Sigma$ be the associated dynamical system as in \eqref{def: difsystem}.  Set
$$D := \max_{1 \leq i \leq n} \deg(f_i), \quad D_{\mathbf{x}} := \max_{1 \leq i \leq n} \deg_{\mathbf{x}}(f_i), \quad D_{\mathbf{u}}:= \max_{1 \leq i \leq n} \deg_{\mathbf{u}}(f_i).$$
Assume $D_\mathbf{x}>0$. Then, for every $h \in \mathbb{K}[\mathbf{x},\mathbf{u}]\setminus \{0\}$, we have that:
\begin{itemize}
    \item $\deg(\mathcal{E}(h^{(s)})) \leq \deg(h) + s (D  -1 ) $,
    \item $\deg_{\mathbf{x}} (\mathcal{E}(h^{(s)})) \leq \deg_{\mathbf{x}}(h)+ s(D_{\mathbf{x}} -1 ) $,
    \item $\deg_{\mathbf{u}^{(\leq \nu)} }(\mathcal{E}(h^{(s)})) \leq \deg_{\mathbf{u}}(h) + sD_\mathbf{u}, $
    \item $\deg_{\mathbf{u}^{(r)}}(\mathcal{E}(h^{(s)})) \leq \displaystyle \left \lfloor \frac{s}{r} \right \rfloor$ for $1 \leq r \leq \nu$.
\end{itemize}
\end{lemma}

\begin{proof}
We argue by induction on $s\in\mathbb{N}_0$, the case $s=0$ being immediate. For the inductive step, we use that differentiation and evaluation in the equations of $\Sigma$ commute up to re-evaluation. More precisely,
   $$ \mathcal{E}(h^{(s)}) = \mathcal{E}\big((\mathcal{E}(h^{(s-1)}))'\big).$$
Thus, $\mathcal{E}(h^{(s)})$ is obtained from $\mathcal{E}(h^{(s-1)})$ by one differentiation followed by evaluation in the equations of $\Sigma$. It is therefore enough to control how the relevant degrees behave under a single differentiation-evaluation step.

By the chain rule, differentiating $\mathcal{E}(h^{(s-1)})$ produces terms of the form
\[
\partial_{\operatorname{coef}}(\mathcal E(h^{(s-1)})), \quad 
\frac{\partial \mathcal{E}(h^{(s-1)})}{\partial x_i}\,x_i',
\quad \text{ and }\quad 
\frac{\partial \mathcal{E}(h^{(s-1)})}{\partial u_j^{(r)}}\,u_j^{(r+1)}.
\]
The terms that come from $\partial_{\operatorname{coef}}(\mathcal E(h^{(s-1)}))$ do not increase any of the degrees since there is no substitution to be made. Moreover, as the degree of the polynomial does not increase when applying $\partial_{\operatorname{coef}}$ and $D_\mathbf{x}\geq1$, it is clear that:
$$\deg(\mathcal{E}(\partial_{\operatorname{coef}
}(\mathcal{E}(h^{(s-1)})))) \leq \deg(h)+ s(D-1),  \quad \deg_{\mathbf{x}}(\mathcal{E}(\partial_{\operatorname{coef}
}(\mathcal{E}(h^{(s-1)})))) \leq \deg_{\mathbf{x}}(h)+ s(D_{\mathbf{x}}-1).$$

In the terms from $\frac{\partial \mathcal{E}(h^{(s-1)})}{\partial x_i}\,x_i'$, after evaluation in the equations of $\Sigma$, each occurrence of $x_i'$ is replaced by $f_i(\mathbf{x}, \mathbf{u})$. Since differentiation with respect to $x_i$ decreases the degree in $\mathbf{x}$ by one, while the substitution $x_i'=f_i(\mathbf{x},\mathbf{u})$ increases it by at most $D_\mathbf{x}$, it follows that the degree with respect to $\mathbf{x}$ increases by at most $D_\mathbf{x}-1$. The same holds for the total degree.

Finally, the terms from $\frac{\partial \mathcal{E}(h^{(s-1)})}{\partial u_j^{(r)}}\,u_j^{(r+1)}$ do not change when $\mathcal{E}$ is applied; so, their degree in $\mathbf{x}$ and total degree are bounded by those of $h^{(s-1)}$.

We conclude that the first and second inequalities in the Lemma hold.

For the degree with respect to $\mathbf{u}^{(\leq \nu)}$, the terms of the form  $\frac{\partial \mathcal{E}(h^{(s-1)})}{\partial u_j^{(r)}}\,u_j^{(r+1)}$
do not increase the degree, whereas the substitution $x_i'=f_ii(\mathbf{x}, \mathbf{u})$ increases it by at most $D_{\mathbf u}$. Hence, one step increases the degree with respect to $\mathbf{u}^{(\leq \nu)}$ by at most $D_{\mathbf u}$. Iterating these estimates yields the  bound.

Finally, to prove the last inequality, notice that, as before, the terms arising from repeatedly applying
\(\partial_{\operatorname{coef}}\) are not those having dominant degree. Therefore, for simplicity, in the following computation we assume that \(\mathbb K\) is a field of constants, that is, \(\partial(\mathbb K)=0\).

Now for  $r\ge 1$ we proceed by induction in $s \in \mathbb{N}_0$. The initial step is clear. Now, by the multivariate Faà di Bruno formula (see \cite[III, Theorem C]{Comtet1974}), every monomial in $h^{(s)}$ is of the form 
$\prod_{j=1}^m \prod_{k= 0}^s \bigl(u_j^{(k)}\bigr)^{\beta_{j,k}}
\cdot
\prod_{i=1}^n \prod_{k=0}^s \bigl(x_i^{(k)}\bigr)^{\gamma_{i,k}}$,
with $\beta_{j,k},\gamma_{i,k}\in\mathbb{N}_0$ satisfying:

$$\displaystyle \sum_{j=1}^m \sum_{k\ge 0} k\,\beta_{j,k}
+
\sum_{i=1}^n \sum_{k\ge 0} k\,\gamma_{i,k}
= s.$$
Hence, applying the preceding estimates to each monomial of \(h^{(s)}\), and using the induction hypothesis for \(\mathcal E(f_i^{(k-1)})\), we obtain:
\begin{align*}\deg_{\mathbf{u}^{(r)}}(\mathcal{E}(h^{(s)})) 
& \leq  \sum_{j=1}^m  \beta_{j,r} + \sum_{i=1}^n \sum_{k =1}^s \gamma_{i,k} \deg_{\mathbf{u}^{(r)}} \left ( \mathcal{E}(f_i^{(k-1)})) \right ) \\ 
& \leq  \sum_{j=1}^m \beta_{j,r} + \sum_{i=1}^n \sum_{k =1}^s \gamma_{i,k} \left \lfloor \frac{k-1}{r} \right  \rfloor \\ & \leq \frac{1}{r}\left ( \sum_{j=1}^m r \beta_{j,r} + \sum_{i=1}^n \sum_{k=1}^{s} \gamma_{i,k}(k-1) \right) \leq \frac{s}{r}.\end{align*}

The result follows since the degree is a nonnegative integer.
\end{proof}

\bigskip

Now we are ready to state and prove our main Theorem. It can be viewed as a differential-parameter analogue of the bounds established in the nonparametric and parametric settings in \cite{mukhina2025projectingdynamicalsystemssupport} and \cite{10.1007/978-3-032-09645-6_15}.

\begin{theorem}\label{mainthm}
Let $g,f_1,\ldots,f_n \in \mathbb{K}[\mathbf{x},\mathbf{u}]$ and  $\Sigma$ be the dynamical system introduced in \eqref{def: difsystem}. Let $\nu:= \operatorname{ord}(\operatorname{Elim}(\Sigma,\mathbf{x}))$. 
Set $d:= \deg(g)$, $d_\mathbf{x}:= \deg_\mathbf{x}(g)$, $d_\mathbf{u} := \deg_\mathbf{u}(g)$, 
$D := \max_{1 \leq i \leq n} \deg(f_i)$, $D_{\mathbf{x}} := \max_{1 \leq i \leq n} \deg_{\mathbf{x}}(f_i)$, $ D_{\mathbf{u}} := \max_{1 \leq i \leq n} \deg_{\mathbf{u}}(f_i).$

Assume $D_\mathbf{x}, d_\mathbf{x} > 0$ and let $\Delta := \displaystyle \prod_{i=0}^\nu \bigl(d + i(D-1)\bigr)$ and $\Delta_\mathbf{x}:= \displaystyle \prod_{i=0}^\nu \bigl(d_\mathbf{x} + i(D_{\mathbf{x}}-1)\bigr)$.
Then, for every monomial $\displaystyle\Big( \prod_{k=0}^\nu \prod_{j=1}^m (u_j^{(k)})^{\ell_{j,k}} \Big)\prod_{i=0}^\nu (y^{(i)})^{e_i}$
appearing in $\operatorname{Elim}(\Sigma,\mathbf{x})$, we have:
\begin{enumerate}
\item $\displaystyle \sum_{k=0}^\nu \sum_{j=1}^m \ell_{j,k} + \sum_{i=0}^\nu \bigl(d + i(D-1)\bigr) e_i \;\le\; \Delta,$
\label{mainthm: desigualdad 1}

\item $\displaystyle  \sum_{i=0}^\nu \bigl(d_\mathbf{x} + i(D_{\mathbf{x}}-1)\bigr) e_i \;\le\; \Delta_\mathbf{x},$
\label{mainthm: desigualdad 2}

\item $\displaystyle \sum\limits_{k=0}^\nu \sum\limits_{j=1}^m \ell_{j,k} + \sum\limits_{i=0}^\nu \bigl(d_\mathbf{u} + iD_\mathbf{u}\bigr) e_i 
\;\le\;
\sum\limits_{l=0}^\nu \frac{(d_\mathbf{u}+l D_\mathbf{u})\Delta_\mathbf{x}}{d_\mathbf{x} + l(D_\mathbf{x}-1)},$
\label{mainthm: desigualdad 3}

\item $\displaystyle \sum\limits_{j=1}^m \ell_{j,r}+ \sum\limits_{i=r}^\nu \left \lfloor\frac{i}{r} \right \rfloor e_i
\;\le\; 
\sum\limits_{l=r}^\nu \frac{\left \lfloor{l}/{r} \right \rfloor \Delta_\mathbf{x}}{d_\mathbf{x}+l(D_\mathbf{x}-1)}, $ for $1\le r \le \nu.$
\label{mainthm: desigualdad 4}
\end{enumerate}
\end{theorem}

\begin{proof}
We prove each inequality by applying Proposition~\ref{Coro: Dandrea} to the hypersurfaces introduced in Proposition~\ref{prop: dim e irred}, and using the degree estimates for the polynomials $\mathcal{E}(g^{(i)})$ provided by Lemma~\ref{lema tecnico}.

To prove the first inequality, we apply Proposition~\ref{Coro: Dandrea} to the variety $V=\mathbb{A}^{n+m(\nu+1)}$ and  the morphism
\[\varphi^{\mathbf{u}}: \mathbb{A}^{n + m (\nu+1)}
 \to \mathbb{A}^{(m+1)(\nu+1)}, \quad
\varphi^{\mathbf{u}}(\mathbf{u}^{(\leq \nu)},\mathbf{x}) =
 \big(\mathbf{u}^{(\leq \nu)}, \mathcal{E}(g), \mathcal{E}(g'),\ldots, \mathcal{E}(g^{(\nu)})\big)
.\]
Since $V$ is an affine space, its degree is equal to $1$. The first $m(\nu+1)$ coordinates of $\varphi^{\mathbf{u}}$ consist of a single $u_j^{(k)}$, so they have degree $1$. On the other hand, Lemma~\ref{lema tecnico} gives the upper bounds
 \[\deg (\mathcal{E}(g^{(i)})) \leq d + i(D-1), \quad i=0, \dots, \nu,\]
  for the total degree of each of the last $\nu+1$ coordinates of $\varphi^{\mathbf{u}}$.
  Substituting these bounds into Proposition~\ref{Coro: Dandrea}, we obtain the first inequality of the statement.

For the remaining inequalities, we apply the same result  to the variety $V=\mathbb{A}^n$, now regarded over the algebraic closure $\widehat{K}$ of the field $\mathbb{K}(\mathbf{u}^{(\leq \nu)})$, and the morphism
\[\varphi : \mathbb{A}^n_{\widehat{K}} \to \mathbb{A}^{\nu+1}_{\widehat{K}},  \quad
    \varphi( \mathbf{x} ) =
 \big(\mathcal{E}(g), \mathcal{E}(g'),\ldots, \mathcal{E}(g^{(\nu)}\big).\]

  In this situation, the  parameters $\mathbf{u}$ and their derivatives are viewed as elements of the coefficient field, and therefore only the degrees with respect to the variables $\mathbf{x}$ have to be considered. Again, since $V$ is an affine space, its degree is $1$, and Lemma~\ref{lema tecnico} provides the required bounds for the $\mathbf{x}$-degrees of the coordinates of $\varphi$:
  \[\deg_{\mathbf{x}}(\mathcal{E}(g^{(i)}))\le d_{\mathbf{x}} + i(D_{\mathbf{x}}-1), \quad i=0, \dots, \nu.\]
  The first inequality in Proposition~\ref{Coro: Dandrea} then yields the second inequality of the Theorem.

The third inequality is obtained by applying the second part of Proposition~\ref{Coro: Dandrea} in the same setting as above, viewing all the variables $\mathbf{u}^{(\leq \nu)}$ as a single parameter block. Since the height of an affine space is zero, the only contribution comes from the degrees of the coordinates $\mathcal{E}(g^{(i)})$ with respect to this parameter block. These are bounded by Lemma~\ref{lema tecnico} as:
\[\deg_{\mathbf{u}^{(\leq \nu)}}(\mathcal{E}(g^{(i)}))\le d_{\mathbf{u}} + i D_{\mathbf{u}}, \quad i=0, \dots, \nu.\]
From these estimates, the second inequality in Proposition~\ref{Coro: Dandrea} yields the desired bound.

Finally, the inequalities of the fourth type are proved in exactly the same way as inequality 3, except that instead of considering $\mathbf{u}^{(\leq \nu)}$ as a single block, we decompose it into the blocks $\mathbf{u},\mathbf{u}',\ldots,\mathbf{u}^{(\nu)}$. We then apply the second part of Proposition~\ref{Coro: Dandrea} to these blocks. The estimates for the corresponding degrees are again given by Lemma~\ref{lema tecnico}, now using the bound for the degree with respect to derivatives of a fixed order: for every $1 \leq r \leq \nu$,
\[\deg_{\mathbf{u}^{(r)}}(\mathcal{E}(g^{(i)})) \leq \displaystyle \left \lfloor \frac{i}{r} \right \rfloor, \quad i=0,\dots, \nu.\]
 The stated inequalities follow for every $1\le r\le \nu$.
 \end{proof}

\bigskip

We now compare our estimates with those obtained in \cite[Theorem~1]{10.1007/978-3-032-09645-6_15}, in order to highlight the differences in the structure of the bounds and the additional information captured by our results.

In \cite{10.1007/978-3-032-09645-6_15},  Mukhina and Pogudin studied the case where all the parameters are constant (that is, when $u_j'=0$ for all $j$). In this situation, they show that every monomial $\displaystyle\prod_{j=1}^m u_j^{\ell_{j}} \prod_{i=0}^\nu (y^{(i)})^{e_i}$ appearing in $\operatorname{Elim}(\Sigma,\mathbf{x})$ satisfies the following inequalities:
\begin{align}
&\sum_{j=1}^m \ell_j + \sum_{i=0}^{\nu}(d_{\mathbf{u}}+iD_{\mathbf{u}})e_i
\le
\sum_{i=0}^{\nu}(d_{\mathbf{u}}+iD_{\mathbf{u}})
\prod_{\substack{j=0\\j\neq i}}^{\nu}\bigl(d_{\mathbf{x}}+j(D_{\mathbf{x}}-1)\bigr), \label{ec 1 Mukhina}\\
&\sum_{j=1}^m \ell_j + \sum_{i=0}^{\nu}(d_{\mathbf{x}}+i(D_{\mathbf{x}}-1))e_i
\le
\prod_{i=0}^{\nu}\bigl(d_{\mathbf{x}}+d_{\mathbf{u}}+i(D_{\mathbf{x}}+D_{\mathbf{u}}-1)\bigr), \label{ec 2 Mukhina}\\
& \sum_{i=0}^{\nu}(d_{\mathbf{x}}+i(D_{\mathbf{x}}-1))e_i
\le
\prod_{i=0}^{\nu}\bigl(d_{\mathbf{x}}+i(D_{\mathbf{x}}-1)\bigr). \label{ec 3 Mukhina}
\end{align}

We observe that inequality~\eqref{ec 2 Mukhina} is redundant. Indeed, adding inequalities~\eqref{ec 1 Mukhina} and~\eqref{ec 3 Mukhina}, we obtain:
\begin{align*}
\sum_{j=1}^m \ell_j + \sum_{i=0}^\nu \bigl(d_{\mathbf{x}}+d_{\mathbf{u}}+i(D_{\mathbf{x}}+D_{\mathbf{u}}-1)\bigr)e_i
&\le \\
\sum_{i=0}^{\nu}(d_{\mathbf{u}}+iD_{\mathbf{u}})
\prod_{\substack{j=0\\j\neq i}}^{\nu}\bigl(d_{\mathbf{x}}+j(D_{\mathbf{x}}-1)\bigr)
&+ \prod_{i=0}^{\nu}\bigl(d_{\mathbf{x}}+i(D_{\mathbf{x}}-1)\bigr).
\end{align*}
Now, the left hand side of this inequality is greater than  or equal to the left hand side of  \eqref{ec 2 Mukhina}, and its right-hand side is bounded above by $\displaystyle \prod_{i=0}^{\nu}\bigl(d_{\mathbf{x}}+d_{\mathbf{u}}+i(D_{\mathbf{x}}+D_{\mathbf{u}}-1)\bigr),$ since all the terms in the sum appear in the expansion of this product. Hence \eqref{ec 2 Mukhina} follows from \eqref{ec 1 Mukhina} and \eqref{ec 3 Mukhina}.

The fact that one of the bounds in \cite{10.1007/978-3-032-09645-6_15} does not provide additional information is related to the fact that the estimates are expressed in terms of separate degree contributions in $\mathbf{x}$ and $\mathbf{u}$.
In contrast, our approach is based on considering the total degree with respect to the variables $(\mathbf{x},\mathbf{u})$. This leads to a different behavior: the analogue of inequality~\eqref{ec 2 Mukhina}, namely \eqref{mainthm: desigualdad 1}, is no longer redundant in general and, moreover, it is a strictly stronger constraint.

\begin{example}\label{rmk: noredundant}
 Consider the system $\Sigma$ from Example \ref{segundo ejemplo}. We have:
\[D=1, \quad  D_\mathbf{x}=1,\quad  D_\mathbf{u}=1,\quad d=2, \quad 
d_\mathbf{x}=1,\quad d_\mathbf{u}=2.
\] The resulting inequalities from Theorem \ref{mainthm} are:
\begin{enumerate}
    \item $\ell_0+\ell_1+\ell_2+2e_0+2e_1+2e_2 \leq 8$

    \item $e_0+e_1+e_2 \leq 1$

    \item $\ell_0+\ell_1+\ell_2+2e_0+3e_1+4e_2 \leq 9$

    \item \begin{enumerate}
		\item $\ell_1+e_1+2e_2 \leq 3$
                \item $\ell_2+e_2 \leq 1.$
\end{enumerate}
\end{enumerate}

In this case none of the inequalities is redundant: the points
\[
(0,9,0,0,0,0),\quad
(0,0,0,0,2,0),\quad
(0,6,0,0,0,1),\quad
(4,0,0,0,0,0),\quad
(0,0,2,0,0,0)
\]
satisfy all the remaining inequalities while violating, respectively, inequalities 1, 2, 3, 4(a) and 4(b), in the order
in which they are displayed.
On the other hand, the polytope defined by these five inequalities contains \(140\) lattice points, whereas, as
we may observe in Example~\ref{segundo ejemplo}, the actual eliminant has only
\(10\) terms. 

This shows that, even when all the inequalities arising from
Theorem~\ref{mainthm} are non-redundant, the resulting polytope may still
overestimate the support of the eliminant significantly.
 \end{example}

In our setting, additional constraints involving derivatives of fixed order of the parameter variables $\mathbf{u}$ naturally arise, see Inequalities \ref{mainthm: desigualdad 4} in Theorem \ref{mainthm}.
These bounds, which control the contribution of derivatives of a fixed order, are new in the sense they do not arise as generalizations of previous inequalities. When 
$$\sum\limits_{l=r}^\nu \frac{\left \lfloor{l}/{r} \right \rfloor}{d_\mathbf{x}+l(D_\mathbf{x}-1)}<1$$
for a fixed $1 \leq r \leq \nu$, we get a constraint in the degree of $\operatorname{Elim}(\Sigma,\mathbf{x})$ in the variables $\mathbf{u}^{(r)}$ that is sharper than the ones from inequalities \eqref{mainthm: desigualdad 1} and \eqref{mainthm: desigualdad 3}. In particular, this is always the case when \(r=\nu\) and either $d_\mathbf{x}>1$ or $D_\mathbf{x} >1$, implying that at least one of all of these new inequalities provides sharper bounds on the degrees in the parameter derivatives under these mild assumptions.

\begin{remark}
Consider a non-autonomous system
\[
\Sigma :=
\begin{cases}
\mathbf{x}' = \mathbf{f}(t,\mathbf{x}, \mathbf{u}), \\
y = g(t,\mathbf{x}, \mathbf{u}),
\end{cases}
\]
where the dynamics and the output depend explicitly on time by means of polynomial functions  $\mathbf{f}(t,\mathbf{x}, \mathbf{u}), g(t,\mathbf{x}, \mathbf{u})$ in $\mathbb{K}[t, \mathbf{x}, \mathbf{u}]$. By introducing a new variable $x_{0} = t$ with $x_{0}' = 1$, this system can be rewritten as
\[
\begin{cases}
x_{0}' = 1, \\
\mathbf{x}' = \mathbf{f}(x_0,\mathbf{x}, \mathbf{u}), \\
y = g(x_0,\mathbf{x}, \mathbf{u}),
\end{cases}
\]
which is a system as in \eqref{def: difsystem}. Hence, Theorem~\ref{mainthm} can be applied to take into account the degrees in $t$ in non-autonomous polynomial systems.
\end{remark}

\subsection{Exploiting the structure of the output: partial evaluation}\label{subsec: output}

The bounds obtained in Theorem~\ref{mainthm} arise from the degree estimates in Lemma~\ref{lema tecnico}, which provide suitable weights for applying Proposition~\ref{Coro: Dandrea}. However, these estimates are obtained after repeatedly differentiating the output equation and fully applying the evaluation morphism $\mathcal{E}$ associated with the system $\Sigma$. In this process, all derivatives are treated uniformly, and the substitutions
\[
x_i^{(k)} \mapsto f_i^{(k-1)}(\mathbf{x},\mathbf{u})
\]
propagate degrees as if every monomial in the successive derivatives $g^{(s)}$ had the maximal possible degree. In particular, this procedure does not take into account cancellations, sparsity patterns, or other structural features that may appear after substitution.

An illustration of this phenomenon appears when the output depends only on a distinguished state variable, namely in systems of the form
\begin{equation}\label{sistema con x1}
    \Sigma:=
    \begin{cases}
        \mathbf{x}' = \mathbf{f}(\mathbf{x},\mathbf{u}),\\
        y = h(x_1),
    \end{cases}
\end{equation}
which naturally arise in applications such as nonlinear control, observability, and identifiability problems. In this setting, the derivatives of the output have a very specific structure, allowing one to exploit this fact in order to obtain improved degree estimates and sharper weighted bounds.

Throughout this subsection, we study this setting in greater depth. We first derive refined inequalities obtained directly from the structure of the output equation and compare them with those arising from Theorem~\ref{mainthm}. We then introduce a partial evaluation procedure which yields further inequalities by keeping certain derivatives as independent variables.

As in the previous section, we set $\nu:=\operatorname{ord}(\operatorname{Elim}(\Sigma, \mathbf{x}))$.

\begin{lemma}\label{lem: degx1}
Let $f_1,\ldots,f_n \in \mathbb{K}[\mathbf{x},\mathbf{u}]$, $h \in \mathbb{K}[x_1] \setminus \mathbb{K}$, and $\Sigma$ be the associated dynamical system as in \eqref{def: difsystem}. 
Set $$D: = \max_{1 \leq i \leq n} \deg(f_i), \quad D_{\mathbf{x}} := \max_{1 \leq i \leq n} \deg_{\mathbf{x}}(f_i), \quad D_{\mathbf{u}} := \max_{1 \leq i \leq n} \deg_{\mathbf{u}}(f_i).$$ 
Then, for $1 \le s \leq \nu$, the following inequalities hold:
\begin{itemize}
    \item $\deg (\mathcal{E}(h^{(s)})) \leq \deg(h) + \deg(f_1)-1 + (s-1) (D-1)$, 
    \item $\deg_{\mathbf{x}} (\mathcal{E}(h^{(s)})) \leq \deg(h) + \deg_\mathbf{x}(f_1)-1 + (s-1) ( D_\mathbf{x}-1) $,
    \item $\deg_{\mathbf{u}^{(\leq \nu)} }(\mathcal{E}(h^{(s)})) \leq \deg_\mathbf{u}(f_1) + (s-1)D_\mathbf{u} $,
    \item $\deg_{\mathbf{u}^{(r)}}(\mathcal{E}(h^{(s)})) \leq \displaystyle \left \lfloor \frac{s-1}{r} \right \rfloor$ for $1 \leq r \leq \nu$.
\end{itemize}
\end{lemma}

\begin{proof}
 If $\deg_{\mathbf{x}}(f_1) = 0$, then $\nu =1$ and the inequalities hold.

Assuming $\deg_{\mathbf{x}}(f_1)>0$, we have $D_\mathbf{x} >0$.  Since $\mathcal{E}(h^{(1)}) = \frac{\partial h}{\partial x_1}(x_1) f_1(\mathbf{x}, \mathbf{u})$, the inequalities hold for $s=1$. For $s>1$, the result follows by applying Lemma~\ref{lema tecnico} to this polynomial.
\end{proof}

\bigskip

Note that, as $\mathcal{E}(h^{(s)})$ does not depend on $\mathbf{u}^{(\nu)}$ for all $1 \leq s \leq \nu$,  the eliminant polynomial does not depend on these variables either. Now, by applying Proposition \ref{Coro: Dandrea} as in Theorem \ref{mainthm} with the bounds from Lemma \ref{lem: degx1} we deduce:

\begin{corollary}\label{Coro: Solo x1}
Let $f_1,\ldots,f_n \in \mathbb{K}[\mathbf{x},\mathbf{u}]$, $h \in \mathbb{K}[x_1] \setminus \mathbb{K}$, and let $\Sigma$ be the associated dynamical system as in \eqref{def: difsystem} with output $ y = h (x_1)$. Let $\nu:= \operatorname{ord}(\operatorname{Elim}(\Sigma,\mathbf{x}))$. 
Set $d:= \deg(h)$, $\delta:= \deg(f_1)$,  $\delta_{\mathbf{x}}:= \deg_{\mathbf{x}}(f_1)$,  $\delta_{\mathbf{u}}:= \deg_{\mathbf{u}}(f_1)$, $D := \max_{1 \leq i \leq n} \deg(f_i)$, $D_{\mathbf{x}} := \max_{1 \leq i \leq n} \deg_{\mathbf{x}}(f_i)$, $ D_{\mathbf{u}} := \max_{1 \leq i \leq n} \deg_{\mathbf{u}}(f_i).$
Assume $d>0, \delta_{\mathbf{x}} >0$. 

Let $\Delta := d \displaystyle \prod_{i=0}^{\nu-1} \bigl(d + \delta -1+i (D-1)\bigr)$ and $\Delta_\mathbf{x}:= d \displaystyle \prod_{i=0}^{\nu-1} \bigl(d+\delta_{\mathbf{x}} -1 +i(D_{\mathbf{x}}-1)\bigr)$.
Then, for every monomial $\displaystyle\Big( \prod_{k=0}^{\nu-1} \prod_{j=1}^m (u_j^{(k)})^{\ell_{j,k}} \Big)\prod_{i=0}^\nu (y^{(i)})^{e_i}$ appearing in $\operatorname{Elim}(\Sigma,\mathbf{x})$, we have:
\begin{enumerate}
\item $\displaystyle \sum_{k=0}^{\nu-1} \sum_{j=1}^m \ell_{j,k}  + d e_0+ \sum_{i=1}^\nu \bigl(d+\delta-1 + (i-1)(D-1)\bigr) e_i \;\le\; \Delta$,
\item $\displaystyle de_0 +\sum_{i=1}^\nu \bigl(d+\delta_\mathbf{x}-1 + (i-1)(D_{\mathbf{x}}-1)\bigr) e_i \;\le\; \Delta_{\mathbf{x}}, $
\item $\displaystyle \sum_{k=0}^{\nu-1} \sum_{j=1}^m \ell_{j,k}  + \sum_{i=1}^\nu \bigl(\delta_\mathbf{u} + (i-1)D_\mathbf{u}\bigr) e_i \;\le\;  \sum_{l=0}^{\nu-1} \frac{(\delta_\mathbf{u}+l D_\mathbf{u}) \Delta_\mathbf{x} }{d+\delta_\mathbf{x}-1 + l (D_\mathbf{x}-1)} , $
\item $\displaystyle \sum_{j=1}^m \ell_{j,r}  + \sum_{i=r+1}^\nu \left \lfloor\frac{i-1}{r} \right \rfloor e_i \;\le\;  \sum_{l=r}^{\nu-1} \frac{\left \lfloor l/r \right \rfloor \Delta_\mathbf{x} }{d+\delta_\mathbf{x}-1+l(D_\mathbf{x}-1)}  $,  for $1 \leq r \leq \nu-1.$
\end{enumerate}
\end{corollary}

\medskip

\begin{remark}
When $h(x_1)=x_1$ and no differential parameters are present,
Corollary~\ref{Coro: Solo x1} recovers the generic sharp bounds from
\cite[Theorems~1 and~3]{mukhina2025projectingdynamicalsystemssupport}
whenever the maximum $D_{\mathbf{x}} := \max_{1 \leq i \leq n} \deg_{\mathbf{x}}(f_i)$ is attained at some $i\ne 1$.

Nevertheless, the full evaluation procedure does not always recover all the sharp
bounds in the extremal case $\delta_\mathbf{x}=D_\mathbf{x}$, since it yields
only one of the corresponding inequalities. This motivates the partial
evaluation method introduced below.
\end{remark}

We now introduce a partial evaluation procedure in which certain derivatives
of the state variables are kept as independent variables, while the remaining
ones are recursively substituted using the equations of $\Sigma$.
This produces refined degree estimates when $\delta_{\mathbf{x}} > \max_{2\le i \le n} \deg_\mathbf{x}(f_i)$ by distinguishing between the different substitution steps appearing in the evaluation process.

Consider the partial evaluation morphism
\[
\mathcal{E}_1:
\mathbb{K}[\mathbf{u}^{(\leq \nu-1)}, \mathbf{x}^{(\leq \nu)}]
\longrightarrow
\mathbb{K}[\mathbf{u}^{(\leq \nu-1)},\mathbf{x},x_1^{(1)}],
\]
which acts as $\mathcal{E}$ except that it keeps $x_1^{(1)}$ unevaluated, and the associated $\mathbb{K}(\mathbf{u}^{(\leq \nu-1)})$-morphism 
\[
\mathcal{E}_1:
\mathbb{K}(\mathbf{u}^{(\leq \nu-1)})[\mathbf{x}^{(\leq \nu)}]
\longrightarrow
\mathbb{K}(\mathbf{u}^{(\leq \nu-1)})[\mathbf{x},x_1^{(1)}].
\]
We define the  varieties
$V \subseteq \mathbb{A}_{\mathbb{K}}^{\nu m+n+1}$ as 
\begin{equation}\label{eq:defV}
V =\Big\{
(\mathbf{u}^{(\leq \nu-1)},\mathbf{x},x_1^{(1)}) \in
\mathbb{A}^{\nu  m+n+1}
\mid x_1^{(1)} = f_1(\mathbf{x},\mathbf{u})
\Big\},
\end{equation}
and $\widehat V \subseteq \mathbb{A}_{\widehat{K}}^{n+1}$, where $\widehat{K}$ is an algebraic closure of $\mathbb{K}(\mathbf{u}^{(\leq \nu-1)})$, as
\begin{equation}\label{eq:defhatV}
\widehat V =\Big\{
(\mathbf{x},x_1^{(1)}) \in
\mathbb{A}_{\widehat{K}}^{n+1}
\mid x_1^{(1)} = f_1(\mathbf{x},\mathbf{u})
\Big\}.
\end{equation}
Note that $V$ and $\widehat V$ are irreducible hypersurfaces in the corresponding affine spaces.

Associated with the previous evaluation maps and varieties, we define 
\begin{equation}\label{eq:defpsi}
\psi:V \rightarrow \mathbb{A}^{ \nu \cdot m + \nu+1}, \quad
\psi(\mathbf{u}^{(\leq \nu-1)},\mathbf{x},x_1^{(1)})
=
(\mathbf{u}^{(\leq \nu-1)},h,\mathcal{E}_1(h'), \ldots,  \mathcal{E}_1(h^{(\nu)})),
\end{equation}
and
\begin{equation}\label{eq:defhatpsi}
\widehat \psi: \widehat{V} \rightarrow \mathbb{A}_{\widehat{K}}^{ \nu+1}, \quad
\widehat \psi(\mathbf{x},x_1^{(1)})
=
(h,\mathcal{E}_1(h'), \ldots,  \mathcal{E}_1(h^{(\nu)})).
\end{equation}

\begin{lemma}\label{lem: psi hypersurface}
In the previous setting, the Zariski closures of $\operatorname{Im}(\psi)$ and $\operatorname{Im}(\widehat \psi)$ are irreducible hypersurfaces, and $\operatorname{Elim}(\Sigma,\mathbf{x})$ is a defining polynomial for each of them.
\end{lemma}

\begin{proof}
Observe that, by definition of $V$, the relation
$x_1^{(1)} = f_1(\mathbf{x},\mathbf{u})$
is imposed. Therefore, when restricting to $V$, the partial evaluation
morphism $\mathcal{E}_1$ behaves as the full evaluation morphism
$\mathcal{E}$ when applied to the successive derivatives of $h(x_1)$.
In particular,
\[
\mathcal{E}_1(h^{(i)}) = \mathcal{E}(h^{(i)}) \quad \text{on } V
\quad \text{for all } 1\le i \le \nu.
\]
It follows that the coordinates of $\psi$ satisfy the same algebraic
relations as in the fully evaluated case. The rest of the proof
proceeds as in Proposition~\ref{prop: dim e irred}, yielding
that $\overline{\operatorname{Im}(\psi)}$ is an irreducible hypersurface and $\operatorname{Elim}(\Sigma,\mathbf{x})$ is its defining polynomial.
Similar arguments imply the result for $\widehat \psi$. 
\end{proof}

\medskip

Now we state a lemma whose proof is routine and analogous to that of Lemma~\ref{lema tecnico}.

\begin{lemma}\label{lema estimaciones}
Let $f_1,\ldots,f_n \in \mathbb{K}[\mathbf{x},\mathbf{u}]$, $h \in \mathbb{K}[x_1] \setminus \mathbb{K}$, and let $\Sigma$ be the associated dynamical system as in \eqref{sistema con x1}. Set $\delta := \deg(f_1)$, $\delta_{\mathbf{x}} := \deg_{\mathbf{x}}(f_1)$, $\delta_{\mathbf{u}} :=  \deg_{\mathbf{u}}(f_1),$ and
$$\mathcal{D} := \max_{2 \leq i \leq n} \deg(f_i), \quad \mathcal {D}_{\mathbf{x}} := \max_{2 \leq i \leq n} \deg_{\mathbf{x}}(f_i), \quad \mathcal{D}_{\mathbf{u}} := \max_{2 \leq i \leq n} \deg_{\mathbf{u}}(f_i).$$
Assume $\delta > \mathcal{D}$, $\delta_{\mathbf{x}} > \mathcal{D}_{\mathbf{x}}$, and $\delta_{\mathbf{u}} > \mathcal{D}_{\mathbf{u}}$.
Then, for $2\le s \le \nu$, the following estimates hold:
\begin{itemize}
    \item $\deg(\mathcal{E}_1(h^{(s)})) \leq \deg(h) + (s-1) (\delta-1) + (\mathcal{D}-1)$,
    \item $\deg_{\mathbf{x},x_1^{(1)}}(\mathcal{E}_1(h^{(s)})) \leq \deg(h) + (s-1) (\delta_{\mathbf{x}}-1) + (\mathcal{D}_{\mathbf{x}} -1)$,
    \item $\deg_{\mathbf{u}^{(\leq \nu)}} (\mathcal{E}_1(h^{(s)})) \leq  (s-1) \delta_\mathbf{u} + \mathcal{D}_\mathbf{u}$.
\end{itemize}
\end{lemma}

Combining Corollary~\ref{Coro: Solo x1} and Proposition~\ref{Coro: Dandrea} applied to the polynomial maps $\psi$ and $\widehat \psi$,  we obtain:

\begin{proposition}\label{proposición solo x1 general}
Let $f_1,\ldots,f_n \in \mathbb{K}[\mathbf{x},\mathbf{u}]$,
$h \in \mathbb{K}[x_1]\setminus\mathbb{K}$, and let $\Sigma$ be the associated
dynamical system as in \eqref{sistema con x1}, with output $y=h(x_1)$. Let
$\nu:=\operatorname{ord}(\operatorname{Elim}(\Sigma,\mathbf{x}))$. 
Set $d:= \deg(h)$, $\delta:= \deg(f_1)$,  $\delta_{\mathbf{x}}:= \deg_{\mathbf{x}}(f_1)$,  $\delta_{\mathbf{u}}:= \deg_{\mathbf{u}}(f_1)$, $\mathcal{D} := \max_{2 \leq i \leq n} \deg(f_i)$, $\mathcal{D}_{\mathbf{x}} := \max_{2 \leq i \leq n} \deg_{\mathbf{x}}(f_i)$, $ \mathcal{D}_{\mathbf{u}} := \max_{2 \leq i \leq n} \deg_{\mathbf{u}}(f_i).$
Assume $\delta > \mathcal{D}$, $\delta_{\mathbf{x}} > \mathcal{D}_{\mathbf{x}}$, and $\delta_{\mathbf{u}} > \mathcal{D}_{\mathbf{u}}$,
and let 
\begin{align*}
\Delta := \displaystyle \prod_{i=0}^\nu \bigl(d + i(\delta-1)\bigr) , &\quad  \Delta_\mathbf{x}:= \displaystyle \prod_{i=0}^\nu \bigl(d + i(\delta_{\mathbf{x}}-1)\bigr),\\
 \widetilde \Delta:= d^2 \prod_{i=2}^{\nu}( d+(i-1)(\delta-1) +\mathcal{D}-1), &\quad 
\widetilde\Delta_{\mathbf{x}}:=d^2 \prod_{i=2}^{\nu} (d+(i-1)(\delta_\mathbf{x}-1)+\mathcal{D}_\mathbf{x}-1).
\end{align*}
Then every monomial $\displaystyle\Big(\prod_{k=0}^{\nu-1}\prod_{j=1}^m
(u_j^{(k)})^{\ell_{j,k}}\Big)
\prod_{i=0}^{\nu}(y^{(i)})^{e_i}$
appearing in $\operatorname{Elim}(\Sigma,\mathbf{x})$ satisfies:

\begin{enumerate}
\item $\displaystyle \sum_{k=0}^{\nu-1}\sum_{j=1}^m \ell_{j,k}
+\sum_{i=0}^{\nu}
\bigl(d+i(\delta-1)\bigr)e_i
\leq
\Delta, $
\label{ecuacion 15}

\item $\displaystyle  \sum_{k=0}^{\nu-1}\sum_{j=1}^m \ell_{j,k}
+d (e_0+e_1)
+\sum_{i=2}^{\nu} (d+(i-1)(\delta-1) +\mathcal{D}-1) e_i
\leq
\delta \widetilde \Delta,$
\label{ecuacion 16}

\item $\displaystyle \sum_{i=0}^{\nu} \bigl(d+i(\delta_\mathbf{x}-1)\bigr)e_i \leq \Delta_{\mathbf{x}}$,
\label{ecuacion 17}

\item $\displaystyle d (e_0+e_1) +\sum_{i=2}^{\nu} (d+(i-1)(\delta_\mathbf{x}-1)+\mathcal{D}_\mathbf{x}-1) e_i
\leq
\delta_\mathbf{x} \widetilde\Delta_\mathbf{x},$
\label{ecuacion 18}

\item $\displaystyle \sum_{k=0}^{\nu-1}\sum_{j=1}^m \ell_{j,k}
+\delta_\mathbf{u}\sum_{i=1}^{\nu} i e_i
\leq
\sum_{l=1}^{\nu} \frac{l\delta_\mathbf{u} \Delta_\mathbf{x}} {d+l(\delta_\mathbf{x}-1)}$,
\label{ecuacion 19}

\item $\displaystyle \sum_{k=0}^{\nu-1}\sum_{j=1}^m \ell_{j,k}
+\sum_{i=2}^{\nu}
\bigl((i-1)\delta_\mathbf{u}+\mathcal{D}_\mathbf{u}\bigr)e_i
\leq \widetilde \Delta_{\mathbf{x}}
\Big(\delta_\mathbf{u} +\delta_\mathbf{x}
\sum_{l=1}^{\nu-1}
\frac{l\delta_\mathbf{u}+\mathcal{D}_\mathbf{u}}{d+l(\delta_\mathbf{x}-1)+\mathcal{D}_\mathbf{x}-1} \Big)$,
\label{ecuacion 20}

\item $\displaystyle \sum_{j=1}^m \ell_{j,r}
+\sum_{i=r+1}^{\nu}
\left\lfloor\frac{i-1}{r}\right\rfloor e_i
\leq
\sum_{l=r+1}^{\nu} \frac{\lfloor(l-1)/{r}\rfloor \Delta_{\mathbf{x}}}
{d+l(\delta_\mathbf{x}-1)}
$, for $1\leq r\leq \nu-1$.
\label{ecuacion 21}
\end{enumerate}
\end{proposition}

\begin{proof}
Inequalities 1, 3, 5, and 7 follow directly from
Corollary~\ref{Coro: Solo x1}, since under our assumptions, 
$\delta = \max_{1\le i \le n} \deg(f_i)$, $\delta_{\mathbf{x}} = \max_{1\le i \le n} \deg_{\mathbf{x}}(f_i)$ and $\delta_{\mathbf{u}} = \max_{1\le i \le n} \deg_{\mathbf{u}}(f_i)$.

To prove the second, fourth and sixth inequalities, we apply  Proposition \ref{Coro: Dandrea} to the mappings $\psi$ and $\widehat{\psi}$ introduced in \eqref{eq:defpsi} and \eqref{eq:defhatpsi}, respectively. Since both $V$ and $\widehat{V}$ are hypersurfaces (see  \eqref{eq:defV} and \eqref{eq:defhatV}), we have that
\begin{equation*}
    \deg(V) = \delta,
\qquad
\deg(\widehat{V}) = \delta_\mathbf{x},
\qquad
h_{\mathbf{u}^{(\leq \nu-1)}}(\widehat V) = \delta_\mathbf{u}.
\end{equation*}
Using these degree and height values, together with the degree estimates from Lemma \ref{lema estimaciones}, the inequalities follow in the same fashion as in Theorem \ref{mainthm}. 
\end{proof}

\bigskip

A natural question that arises from Proposition~\ref{proposición solo x1 general} is whether all the obtained inequalities are non-redundant or some of them can be derived from the others (c.f. the discussion after the proof of Theorem \ref{mainthm}). We illustrate the non-redundant behaviour of our result with two complementary examples. In both cases we take $\nu=n=2$. In each case, we fix a degree pattern and consider the  inequalities stated in the Proposition for those degrees. In the tables below, the second column contains the number of monomials with exponents satisfying the active inequalities: the first line corresponds to the entire system, while in the remaining ones, we remove one inequality at each time. 

\begin{center}
\setlength{\tabcolsep}{4pt}
\renewcommand{\arraystretch}{1}
\begin{minipage}{0.5\textwidth}
\centering
\begin{tabular}{|l|r|}
\hline
\multicolumn{2}{|c|}{\textbf{Example A}} \\
\hline
\multicolumn{2}{|c|}{$
\mathcal{D}=3,\ 
\mathcal{D}_{\mathbf{x}}=2,\
\mathcal{D}_{\mathbf{u}}=2
$} \\
\multicolumn{2}{|c|}{$
\delta=5,\
\delta_{\mathbf{x}}=3,\
\delta_{\mathbf{u}}=4,\
d=1
$} \\
\hline
\textbf{Active inequalities} & \textbf{Monomials} \\
\hline
All inequalities & 7152 \\
Without \eqref{ecuacion 15} & 7317 \\
Without \eqref{ecuacion 16} & 7922 \\
Without \eqref{ecuacion 17} & 9749 \\
Without \eqref{ecuacion 18} & 7628 \\
Without \eqref{ecuacion 19} & 7152 \\
Without \eqref{ecuacion 20} & 7156 \\
Without \eqref{ecuacion 21} & 29844 \\
\hline
\end{tabular}
\end{minipage}
\hfill
\begin{minipage}{0.49\textwidth}
\centering
\begin{tabular}{|l|r|}
\hline
\multicolumn{2}{|c|}{\textbf{Example B}} \\
\hline
\multicolumn{2}{|c|}{$
\mathcal{D} =1,\
\mathcal{D}_{\mathbf{x}}=1,\
\mathcal{D}_{\mathbf{u}}=1
$} \\
\multicolumn{2}{|c|}{$
\delta=4,\
\delta_{\mathbf{x}}=2,\
\delta_{\mathbf{u}}=2,\
d=1
$} \\
\hline
\textbf{Active inequalities} & \textbf{Monomials} \\
\hline
All inequalities & 464 \\
Without \eqref{ecuacion 15} & 464 \\
Without \eqref{ecuacion 16} & 464 \\
Without \eqref{ecuacion 17} & 519 \\
Without \eqref{ecuacion 18} & 569 \\
Without \eqref{ecuacion 19} & 470 \\
Without \eqref{ecuacion 20} & 563 \\
Without \eqref{ecuacion 21} & 1032 \\
\hline
\end{tabular}
\end{minipage}
\end{center}

In Example A, all inequalities except
\eqref{ecuacion 19} are non-redundant, and Example B shows that
\eqref{ecuacion 19} is also non-redundant. Therefore, no
inequality in Proposition~\ref{proposición solo x1 general} can be removed
from the system of bounds.

\medskip

On the other hand, in the nonparametric setting, the relevant inequalities in
Proposition~\ref{proposición solo x1 general} are those involving only
degrees in the variables $\mathbf{x}$, namely inequalities  \eqref{ecuacion 17} and \eqref{ecuacion 18}, which are both valid under the assumption
$\delta_\mathbf{x}>\mathcal D_\mathbf{x}\geq0$. For
$n=\nu=2$, these inequalities become:
\begin{align*}
     d e_0 +  (d+\delta_{\mathbf{x}}-1)e_1 + (d+ 2\delta_\mathbf{x}-2)e_2 & \leq d (d+\delta_{\mathbf{x}}-1)(d + 2\delta_\mathbf{x}-2), \\
    d e_0 + d e_1 + (d+ \mathcal{D}_{\mathbf{x}}+ \delta_\mathbf{x}-2)e_2 & \leq \delta_\mathbf{x} d^2(d+ \mathcal{D}_{\mathbf{x}}+ \delta_\mathbf{x}-2).
\end{align*}
When $\mathcal{D}_\mathbf{x}=d=1$, these inequalities recover the Newton polytope of the eliminant of a generic system  (see \cite[Theorem 2]{mukhina2025projectingdynamicalsystemssupport} or Example \ref{Ejemplo}) and whenever only $d=1$, they give a sharp approximation of the actual Newton polytope of the eliminant, as the following table illustrates: 

\begin{center}
\begin{tabular}{|c|c|c|c|}
\hline
$[\delta_\mathbf{x},\mathcal{D}_\mathbf{x}]$&Proposition \ref{proposición solo x1 general} & Newton Polytope & \% \\
\hline
$[2,1]$ & 19  & 19 & 100\%\\
$[3,1]$ & 67 & 67 & 100\%\\
$[3,2]$ & 80 & 77 & 96\% \\
$[4,2]$ & 210 & 201 & 95\% \\
$[8,3]$ & 2758  & 2617 & 95\% \\
$[15,2]$ & 28480 & 27921 & 98\% \\
$[26,3]$ & 250183 & 244531 & 97\% \\
\hline
\end{tabular}
\end{center}

In the second column, we count the number of possible monomials in $\operatorname{Elim}(\Sigma, \mathbf{x})$ according to our results for a system with  $[\delta_{\mathbf{x}}, \mathcal{D}_{\mathbf{x}}]$ as in the first column. The third column shows the actual number of points in the Newton polytope of the eliminant of a system with generic coefficients with degrees $[\delta_\mathbf{x},D_\mathbf{x}]$ as in the first column, which is characterized in \cite[Theorem 2]{mukhina2025projectingdynamicalsystemssupport}.

\section{Experimental results} \label{sec:experiments}

In this section we illustrate the bounds obtained in Theorem~\ref{mainthm}
by means of computational experiments. All experiments were carried out on a standard workstation (AMD Ryzen 7 7840HS, 32\,GB RAM, AMD Radeon 780M Graphics (3 GB)).

We consider polynomial dynamical systems $\Sigma$ as in \eqref{def: difsystem}.
We use the notation for degrees introduced in Theorem~\ref{mainthm}, namely
\begin{align*}
D_{\mathbf{x}} = \max_{1 \leq i \leq n} \deg_{\mathbf{x}}(f_i), \quad & D_{\mathbf{u}}= \max_{1 \leq i \leq n} \deg_{\mathbf{u}}(f_i), \quad D = \max_{1 \leq i \leq n} \deg(f_i), \\
d_\mathbf{x}= \deg_\mathbf{x}(g), \quad & d_\mathbf{u}= \deg_\mathbf{u}(g), \quad d= \deg(g).
\end{align*}

\subsection{Algorithm and strategy}

We first recall that the case without parameters $\mathbf{u}$, namely \(m=0\), was already studied in \cite{10.1007/978-3-032-09645-6_15}, where extensive computational experiments showed that the bound in Theorem \ref{mainthm} — which, in this case, coincides with the bound in \cite[Corollary 1]{10.1007/978-3-032-09645-6_15} — is sharp, in the sense that it correctly predicts the Newton polytope of the corresponding eliminant.

We now turn to the case where parameters \(\mathbf{u}\) are present. In this setting, direct computations over fields of characteristic zero quickly become infeasible due to the rapid growth of the eliminant. Although in the non-parametric case this difficulty can still be handled by standard Gröbner basis elimination in \textsc{Singular}, the introduction of parameters leads to a substantial increase in complexity.

To overcome these limitations, we extended to the differential-parameter setting the evaluation--interpolation algorithm due to Mukhina and Pogudin (see \cite{10.1007/978-3-032-09645-6_15}, \cite{mukhina2025projectingdynamicalsystemssupport}, \cite{10.1145/3747199.3747564}), originally developed for systems with constant parameters, and implemented it in \textsc{Julia}. Following their approach, we construct a candidate support using the degree bounds developed in the previous sections and reconstruct the eliminant by interpolation over a finite field.
This reconstruction procedure is summarized in Algorithm \ref{alg:interpolation}, which follows the approach described in \cite[Section 8]{10.1007/978-3-032-09645-6_15}, based on \cite[Algorithm 1]{mukhina2025projectingdynamicalsystemssupport}.

\begin{algorithm}[H] 
\small{
\caption{Evaluation--Interpolation Eliminant Reconstruction}
\label{alg:interpolation}
\begin{algorithmic}[1]
\State \textbf{Input:} ODE system $\Sigma$, prime number $p$
\State \textbf{Output:} Eliminant $E$ modulo $p$
\State Compute the degree data $(D_\mathbf{x},D_\mathbf{u},D,d_\mathbf{x},d_\mathbf{u},d)$.
\State Construct the candidate support $S$ using the inequalities from Theorem \ref{mainthm}.
\State Reduce the system modulo $p$.
\State For $i=0,\dots, n$, compute $g_i:=\mathcal{E}(g^{(i)})$.

\State Choose \(N:=\#S\) random points $\xi_1,\ldots,\xi_N$
in the space with coordinates \((\mathbf{x},\mathbf{u}^{(\leq n)})\)
\State For $1\le k\le N$, evaluate

$\displaystyle E\bigl(\mathbf{u}^{(\leq n)}(\xi_k),g_{0}(\xi_k),\ldots,g_{n}(\xi_k)\bigr)
    =
    \sum\limits_{I\in S}\alpha_I
    \bigl(\mathbf{u}^{(\leq n)}(\xi_k),g_0(\xi_k),\ldots,g_n(\xi_k)\bigr)^I$.

\noindent Let $M$ be the $N\times N$ matrix of the resulting homogeneous linear system in the coefficients \(\alpha_I\). 

\State Compute a basis \(\{v_1,\ldots,v_r\}\) of \(\ker(M)\). 
\If{\(r=1\)}
    \State Recover a candidate eliminant \(E\) from \(v_1\).
\ElsIf{\(r>1\)}
    \State Recover the polynomials \(E_1,\ldots,E_r\) associated with  \(v_1,\ldots,v_r\).
    \State Choose $E$ from $\{ E_1,\ldots,E_r\}$ of minimal order and, among them, of minimal total degree.
\Else
    \State \Return failure.
\EndIf
\State $\mathsf{verification}:=
\mathsf{Verify}\!\left(
E\bigl(\mathbf{u}^{(\leq n)},
\mathcal{E}(g^{(0)}),\ldots,\mathcal{E}(g^{(n)})\bigr)=0
\right)$.
\If{$\mathsf{verification}=\mathrm{true}$}
    \State \Return $E$
\Else
    \State \Return failure.
\EndIf
\end{algorithmic}}
\end{algorithm}

\subsection{Computations}

For each choice of degrees:
\[
    [D_{\mathbf{x}},D_{\mathbf{u}},D], \qquad
    [d_{\mathbf{x}},d_{\mathbf{u}},d],
\]
we construct systems with the largest possible support in
\((\mathbf{x},\mathbf{u})\) and $n=2$ state variables $x_1,x_2$ compatible with these degree constraints.
We repeat the experiments for different random choices of the generic system
\(\Sigma\) and for increasing primes \(p\). The reported support sizes correspond
to the outcomes that remain stable under these variations.

To illustrate the practical impact of the support bounds in the evaluation-interpolation procedure efficiency, consider dense systems and degree data
\[
[D_\mathbf{x},D_\mathbf{u},D] = [2,2,2],
\qquad
[d_\mathbf{x},d_\mathbf{u},d] = [2,2,2].
\]
For a system with one constant parameter, Theorem~\ref{mainthm} predicts a support consisting of $1292$ monomials, which coincides with the exact support of the eliminant. In contrast, the support generated by the implementation of Mukhina and Pogudin's method in \texttt{DiffMinPoly} contains $12007$ monomials.

The difference becomes even more pronounced when two constant parameters are present. In this case, our bound still predicts the exact support size of $8189$ monomials, whereas the support implemented in \texttt{DiffMinPoly} (provided by \cite[Theorem 1]{10.1007/978-3-032-09645-6_15}) contains $437228$ monomials. Moreover, the gap between the two supports increases further as the number of parameters grows.

Since the evaluation-interpolation algorithm constructs a square interpolation matrix whose size equals the cardinality of the candidate support, the quality of the candidate support  has a direct impact on the complexity of the computation. The examples discussed above, together with the tables below, show that, in the constant-parameter setting, the bounds of Theorem~\ref{mainthm} lead to significantly smaller interpolation problems than the bounds from \cite[Theorem 1]{10.1007/978-3-032-09645-6_15} and  provide an extremely accurate description of the eliminant support.

\medskip

\begin{table}[H]
\centering
\begin{tabular}{|c|c|c|c|c|}
\hline
$[D_\mathbf{x},D_\mathbf{u},D]$ & $[d_\mathbf{x},d_{\mathbf{u}},d]$
& \multicolumn{2}{c|}{\# of terms} & \% \\
\cline{3-4}
 & & Theorem~\ref{mainthm} & $\operatorname{Elim}(\Sigma,\mathbf{x})$ & \\
\hline
$[1,1,1]$ & $[1,1,1]$ & 5 & 5 & 100\% \\
$[1,1,1]$ & $[2,2,2]$ & 105 & 105 & 100\% \\
$[1,1,1]$ & $[3,3,3]$ & 1705 & 1705 & $100\%$ \\
$[1,1,1]$ & $[4,4,4]$ & 16473 & 16473 & $100\%$ \\
\hline
$[2,1,2]$ & $[2,0,2]$ & 1282 & $ 1222$ & $ 95\%$ \\
$[3,3,3]$ & $[1,1,1]$ & 469 & 469 & 100\% \\
$[2,2,2]$ & $[1,1,1]$ & 64 & 64 & 100\% \\
$[2,2,2]$ & $[2,2,2]$ & 1292 & $1292$ & $ 100\%$ \\
$[3,1,3]$ & $[2,0,2]$ & 7266 & 6752 & $93\%$ \\
$[3,2,3]$ & $[2,1,2]$ & 7875 & 7805 & $99\%$ \\
$[3,3,3]$ & $[2,2,2]$ & 7875 & 7875 & $100\%$ \\

\hline
\end{tabular}
\label{tabla 1}\caption*{Constant-parameter case ($m=1$, $\mathbf{u}'=0$)}
\end{table}

\begin{table}[H]
\centering
\begin{tabular}{|c|c|c|c|c|}
\hline
$[D_\mathbf{x},D_\mathbf{u},D]$ & $[d_\mathbf{x},d_\mathbf{u},d]$
& \multicolumn{2}{c|}{\# of terms} & \% \\
\cline{3-4}
 & & Theorem~\ref{mainthm} & $\operatorname{Elim}(\Sigma,\mathbf{x})$ & \\
\hline
$[1,1,1]$ & $[1,1,1]$ & 6 & 6 & 100\% \\
$[1,1,1]$ & $[2,2,2]$ & 295 & 295 & 100\% \\
$[2,2,2]$ & $[1,1,1]$ & 155 & 155 & $100\%$ \\
$[2,2,2]$ & $[2,2,2]$ & 8189 & 8189 & $100\%$ \\

\hline
\end{tabular}
\label{tabla 2}\caption*{Constant-parameter case ($m=2$, $\mathbf{u}'=0$)}
\end{table}

A remark we can extract from the computed data is that in the constant parameter case, and for sufficiently dense systems ($d_{\mathbf{x}} = d_{\mathbf{u}} =  d$, $D_{\mathbf{x}} = D_{\mathbf{u}} = D$),
the inequality given in Theorem~\ref{mainthm},
\[
\sum_{j=1}^m \ell_j
+ \sum_{i=0}^\nu \bigl(d + i(D-1)\bigr) e_i
\;\le\;
\prod_{i=0}^\nu \bigl(d + i(D-1)\bigr),
\]
appears to describe exactly the Newton polytope of the eliminant.  

\bigskip

 We now turn to the differential-parameter setting. In contrast with the constant-parameter case, derivatives of the inputs contribute additional variables to the eliminant, causing a much faster growth of the candidate support and making the reconstruction problem substantially more challenging. The following experiments illustrate both the effectiveness and the limitations of the proposed bounds in this more general situation.

\begin{table}[H]
\centering
\begin{tabular}{|c|c|c|c|c|}
\hline
$[D_\mathbf{x},D_\mathbf{u},D]$ & $[d_\mathbf{x},d_\mathbf{u},d]$
& \multicolumn{2}{c|}{\# of terms} & \% \\
\cline{3-4}
 & & Theorem~\ref{mainthm} & $\operatorname{Elim}(\Sigma,\mathbf{x})$ & \\
\hline
$[1,1,1]$ & $[1,0,1]$ & 7 & 6 & 85\% \\
$[1,0,1]$ & $[1,1,1]$ & 7 & 7 & 100\% \\
$[1,1,1]$ & $[1,1,1]$ & 7 & 7 & 100\% \\
\hline
$[1,1,1]$ & $[2,0,2]$ & 697 & 254 & 36\% \\
$[1,1,1]$ & $[2,1,2]$ & 697 & 655 & 94\% \\
$[1,0,1]$ & $[2,2,2]$ & 697 & 675 & 98\% \\
$[1,1,1]$ & $[2,2,2]$ & 697 & 675 & 98\% \\
\hline
$[2,2,2]$ & $[1,1,1]$ & 292 & 115 & 39\% \\
$[3,3,3]$ & $[1,1,1]$ & 5722 & 1144 & 20\% \\
$[3,2,3]$ & $[1,1,1]$ & 5722 & 1144 & 20\% \\
$[3,3,3]$ & $[1,0,1]$ & 5722 & 865 & 15\% \\
$[2,2,2]$ & $[2,1,2]$ & 33779 & 9717 & 29\% \\
$[2,1,2]$ & $[2,2,2]$ & 33779 & 9971 & 30\% \\
$[2,2,2]$ & $[2,2,2]$ & 33779 & 9971 & 30\% \\
\hline
\end{tabular}
\caption*{Differential-parameter case ($m=1$)}
\end{table}

The differential-parameter experiments show that the bounds from Theorem \ref{mainthm} are not as sharp as those obtained in the constant parameter setting. Moreover, the table above shows not only that our candidate support grows as the degrees increase, but also that the size of the eliminant grows as well. As a result, it becomes computationally challenging  to determine how sharp the bounds from Theorem \ref{mainthm} are asymptotically as the degrees grow.

\section{Conclusions and future work}\label{sec:conclusions}

The theoretical results presented in this paper (see Section \ref{sec 2}) provide the first descriptions of the support of input-output equations for differential-algebraic dynamical systems with differential inputs of the form \eqref{ecuacion sistema introduccion}. Our approach is based on an analogue of Perron's elimination theorem, that relies on the notion of height of algebraic varieties  over a function field developed in \cite{D2013}, which allows us to recover previous results by Mukhina and Pogudin (\cite{10.1007/978-3-032-09645-6_15},  \cite{mukhina2025projectingdynamicalsystemssupport}) and extend them to the setting where differential parameters are present.

The computational experiments presented in Section \ref{sec:experiments} show that implementations based on classical algebraic elimination algorithms are severely limited in this context, due to the rapid growth of both the candidate supports and the eliminant polynomials. Our computations suggest that evaluation--interpolation methods provide a more suitable framework for computing eliminant polynomials. 

In the constant-parameter setting, the support bounds developed in this work lead to highly efficient reconstructions and, in all tested examples, accurately predict the exact support of the eliminant.
The differential-parameter setting turns out to be more challenging, and our estimates in this direction are noticeably less sharp. This can be explained by the fact that derivatives introduce highly structured and non-generic degree patterns as they appear iteratively. It is reasonable to expect that a behavior similar to the constant-parameter case would emerge if one considers systems where the parameters and their derivatives are treated more symmetrically. It remains an open question to obtain more accurate descriptions of the support of eliminant polynomials in the differential-parameter setting.

Our support bounds are stated in terms of the degrees of the system, but neither Theorem \ref{mainthm} nor Proposition \ref{proposición solo x1 general} take sparsity patterns into account. Some computational results in this direction can be found in \cite{10.1145/3747199.3747564}. A further analysis of the problem in the sparse setting (particularly, a better understanding of the structure of successive derivatives when the equations contain differential parameters) is needed to prove general sharp results following this approach.

Finally, we point out that when differential inputs are considered, the size of the candidate support, and hence of the corresponding interpolation matrices, grows rapidly and eventually becomes the main computational bottleneck.
This suggests that further progress from the algorithmic viewpoint will require not only sharper support bounds, but also more refined evaluation-interpolation strategies.

\section{Declaration of generative AI and AI-assisted technologies in the manuscript preparation process}
During the preparation of this work, the authors used OpenAI’s ChatGPT to assist with the implementation of Algorithm 1. After using this tool, the authors reviewed and edited the resulting content as needed and take full responsibility for the content of the published article.

\printbibliography

@article{D2013,
  author       = {D’Andrea, Carlos and Krick, Teresa and Sombra, Martín},
  title        = {Heights of varieties in multiprojective spaces and arithmetic Nullstellensätze},
  journal      = {Annales scientifiques de l'École Normale Supérieure},
  volume       = {46},
  number       = {4},
  pages        = {549--627},
  year         = {2013},
  url          = {http://eudml.org/doc/272167},
  language     = {eng},
}

@misc{mukhina2025projectingdynamicalsystemssupport,
  author       = {Yulia Mukhina and Gleb Pogudin},
  title        = {Projecting dynamical systems via a support bound},
  year         = {2025},
  eprint       = {2501.13680},
  eprinttype   = {arXiv},
  eprintclass  = {cs.SC},
}

@article{Heintz1983,
  author       = {Heintz, Joos},
  title        = {Definability and fast quantifier elimination in algebraically closed fields},
  journal      = {Theoretical Computer Science},
  volume       = {24},
  number       = {3},
  pages        = {239--277},
  year         = {1983},
  doi          = {10.1016/0304-3975(83)90002-6},
}

@book{harris1992algebraic,
  author       = {Harris, Joe},
  title        = {Algebraic Geometry: A First Course},
  series       = {Graduate Texts in Mathematics},
  volume       = {133},
  publisher    = {Springer-Verlag, New York},
  year         = {1992},
}

@book{mumford1976algebraic,
  author       = {Mumford, David},
  title        = {Algebraic Geometry I: Complex Projective Varieties},
  series       = {Classics in mathematics},
  volume       = {221},
  publisher    = {Springer-Verlag, Berlin-New York},
  year         = {1976},
}

@article{GKO2016,
  author       = {Gustavson, R. and Kondratieva, M. and Ovchinnikov, A.},
  title        = {New effective differential {Nullstellensatz}},
  journal      = {Advances in Mathematics},
  volume       = {290},
  pages        = {1138--1158},
  year         = {2016},
  doi          = {10.1016/j.aim.2015.12.021},
  language     = {English},
}

@book{ritt1950differential,
  author       = {Ritt, Joseph Fels},
  title        = {Differential Algebra},
  series       = {American Mathematical Society Colloquium Publications},
  volume       = {33},
  publisher    = {American Mathematical Society},
  address      = {New York},
  year         = {1950},
}

@book{kolchin1973differential,
  author       = {Kolchin, Ellis R.},
  title        = {Differential Algebra and Algebraic Groups},
  publisher    = {Academic Press},
  address      = {New York},
  year         = {1973},
}

@article{Jelonek2005,
  author       = {Jelonek, Zbigniew},
  title        = {On the effective {Nullstellensatz}},
  journal      = {Inventiones Mathematicae},
  volume       = {162},
  number       = {1},
  pages        = {1--17},
  year         = {2005},
  doi          = {10.1007/s00222-004-0434-8},
}

@book{Comtet1974,
  author       = {Comtet, Louis},
  title        = {Advanced Combinatorics: The Art of Finite and Infinite Expansions},
  edition      = {1},
  publisher    = {D. Reidel Publishing Co. / Springer},
  address      = {Dordrecht, Holland},
  pages        = {xi+343},
  year         = {1974},
  doi          = {10.1007/978-94-010-2196-8},
}

@inproceedings{10.1007/978-3-032-09645-6_15,
  author       = {Mukhina, Yulia and Pogudin, Gleb},
  editor       = {Boulier, Fran{\c{c}}ois and Mou, Chenqi and Sadykov, Timur M. and Vorozhtsov, Evgenii V.},
  title        = {Support Bound for Differential Elimination in Polynomial Dynamical Systems},
  booktitle    = {Computer Algebra in Scientific Computing},
  publisher    = {Springer Nature Switzerland},
  address      = {Cham},
  pages        = {265--285},
  year         = {2026},
}

@inproceedings{10.1145/3747199.3747564,
  author       = {Mohr, Rafael and Mukhina, Yulia},
  title        = {On the Computation of {Newton} Polytopes of Eliminants},
  booktitle    = {Proceedings of the 2025 International Symposium on Symbolic and Algebraic Computation},
  series       = {ISSAC '25},
  publisher    = {Association for Computing Machinery},
  address      = {New York, NY, USA},
  pages        = {215--223},
  year         = {2025},
  doi          = {10.1145/3747199.3747564},
}

@inproceedings{Fliess1995Implicit,
  author       = {Michel Fliess and Jean L{\'e}vine and Philippe Martin and Pierre Rouchon},
  title        = {Implicit Differential Equations and Lie--B{\"a}cklund Mappings},
  booktitle    = {Proceedings of the 34th IEEE Conference on Decision and Control},
  address      = {New Orleans, Louisiana, USA},
  pages        = {2704--2709},
  year         = {1995},
}

@article{doi:10.1137/22M1469067,
  author       = {Dong, Ruiwen and Goodbrake, Christian and Harrington, Heather A. and Pogudin, Gleb},
  title        = {Differential Elimination for Dynamical Models via Projections with Applications to Structural Identifiability},
  journal      = {SIAM Journal on Applied Algebra and Geometry},
  volume       = {7},
  number       = {1},
  pages        = {194--235},
  year         = {2023},
  doi          = {10.1137/22M1469067},
}

@article{DALFONSO2014588,
  author       = {Lisi D’Alfonso and Gabriela Jeronimo and Pablo Solernó},
  title        = {Effective differential {Nullstellensatz} for ordinary {DAE} systems with constant coefficients},
  journal      = {Journal of Complexity},
  volume       = {30},
  number       = {5},
  pages        = {588--603},
  year         = {2014},
  doi          = {10.1016/j.jco.2014.01.001},
}

@incollection{Boulier+2007+109+138,
  author       = {François Boulier},
  editor       = {Markus Rosenkranz and Dongming Wang},
  title        = {Differential Elimination and Biological Modelling},
  booktitle    = {Gr{\"o}bner Bases in Symbolic Analysis},
  publisher    = {De Gruyter},
  address      = {Berlin, Boston},
  pages        = {109--138},
  year         = {2007},
  doi          = {10.1515/9783110922752.109},
}

@article{10.1093/imrn/rnaa302,
    author = {Ovchinnikov, Alexey and Pogudin, Gleb and Vo, Thieu N},
    title = {Bounds for Elimination of Unknowns in Systems of Differential-Algebraic Equations},
    journal = {International Mathematics Research Notices},
    volume = {2022},
    number = {16},
    pages = {12342-12377},
    year = {2022},
    month = {08},
    issn = {1073-7928},
    doi = {10.1093/imrn/rnaa302},

}

@article{Ovchinnikov2023,
  author  = {Ovchinnikov, Alexey and Pogudin, Gleb and Thompson, Peter},
  title   = {Parameter Identifiability and Input--Output Equations},
  journal = {Applicable Algebra in Engineering, Communication and Computing},
  year    = {2023},
  volume  = {34},
  number  = {2},
  pages   = {165--182},
  doi     = {10.1007/s00200-021-00486-8},
}

@article{Rueda2016,
 author = {Rueda, Sonia L.},
 title = {Differential elimination by differential specialization of {Sylvester} style matrices},
 fjournal = {Advances in Applied Mathematics},
 journal = {Adv. Appl. Math.},
 volume = {72},
 pages = {4--37},
 year = {2016},
  doi = {10.1016/j.aam.2015.07.002},
}

@incollection{Hubert2003,
 author = {Hubert, Evelyne},
 title = {Notes on triangular sets and triangulation-decomposition algorithms. {II}: {Differential} systems},
 booktitle = {Symbolic and numerical scientific computation. Second international conference, SNSC 2001, Hagenberg, Austria, September 12--14, 2001. Revised papers},
 pages = {40--87},
 year = {2003},
 publisher = {Berlin: Springer},
 doi = {10.1007/3-540-45084-x_2},
}

@article{DJS2006,
 author = {D'Alfonso, Lisi and Jeronimo, Gabriela and Solern{\'o}, Pablo},
 title = {On the complexity of the resolvent representation of some prime differential ideals},
 fjournal = {Journal of Complexity},
 journal = {J. Complexity},
 volume = {22},
 number = {3},
 pages = {396--430},
 year = {2006},
 doi = {10.1016/j.jco.2005.10.002},
}

@incollection{CarraFerro2007,
 author = {Carr{\`a} Ferro, Giuseppa},
 title = {A survey on differential {Gr{\"o}bner} bases},
 booktitle = {Gr\"obner bases in symbolic analysis. Based on talks delivered at the special semester on Gr\"obner bases and related methods, Linz, Austria, May 2006},
 pages = {77--108},
 year = {2007},
 publisher = {Berlin: Walter de Gruyter},
 }

@incollection{Grigorev1989,
 author = {Grigor'ev, D. Yu.},
 title = {Complexity of quantifier elimination in the theory of ordinary differential equations},
 booktitle = {EUROCAL '87. European Conference on Computer Algebra, Leipzig, GDR, June 2--5, 1987. Proceedings.},
 pages = {11--25},
 year = {1989},
 publisher = {Berlin etc.: Springer-Verlag},
 doi = {10.1007/3-540-51517-8_81},
}

@article{Sedoglavic2002,
 author = {Sedoglavic, Alexandre},
 title = {A probabilistic algorithm to test local algebraic observability in polynomial time},
 fjournal = {Journal of Symbolic Computation},
 journal = {J. Symb. Comput.},
 volume = {33},
 number = {5},
 pages = {735--755},
 year = {2002},
 doi = {10.1006/jsco.2002.0532},
}

@incollection{MRS2018,
 author = {Meshkat, Nicolette and Rosen, Zvi and Sullivant, Seth},
 title = {Algebraic tools for the analysis of state space models},
 booktitle = {The 50th anniversary of Gr\"obner bases. Proceedings of the 8th Mathematical Society of Japan-Seasonal Institute (MSJ-SI 2015), Osaka, Japan, July 1--10, 2015},
 pages = {171--205},
 year = {2018},
 publisher = {Tokyo: Mathematical Society of Japan (MSJ)},
 language = {English},
}

@article{JS2023,
 author = {Jeronimo, Gabriela and Solern{\'o}, Pablo},
 title = {Weak identifiability for differential algebraic systems},
 fjournal = {Advances in Applied Mathematics},
 journal = {Adv. Appl. Math.},
 volume = {147},
 pages = {31},
 note = {Id/No 102519},
 year = {2023},
 doi = {10.1016/j.aam.2023.102519},
 }

@book{Ritt1932,
 author = {Ritt, Joseph Fels},
 title = {Differential equations from the algebraic standpoint},
 fseries = {Colloquium Publications. American Mathematical Society},
 series = {Colloq. Publ., Am. Math. Soc.},
 volume = {14},
 year = {1932},
 publisher = {American Mathematical Society (AMS), Providence, RI},
 }

@incollection{Boulier-etal1995,
 author = {Boulier, Fran{\c{c}}ois and Lazard, Daniel and Ollivier, Fran{\c{c}}ois and Petitot, Michel},
 title = {Representation for the radical of a finitely generated differential ideal},
 booktitle = {Proceedings of the 1995 international symposium on symbolic and algebraic computation, ISSAC '95, Montreal, Canada, July 10--12, 1995},
 pages = {158--166},
 year = {1995},
 publisher = {New York, NY: ACM Press},
}

@article{Li-etal2015,
 author = {Li, Wei and Yuan, Chun-Ming and Gao, Xiao-Shan},
 title = {Sparse differential resultant for {Laurent} differential polynomials},
 fjournal = {Foundations of Computational Mathematics},
 journal = {Found. Comput. Math.},
 volume = {15},
 number = {2},
 pages = {451--517},
 year = {2015},
 doi = {10.1007/s10208-015-9249-9},
}

\end{document}